\documentclass[11pt]{article}
\usepackage{geometry}                
\usepackage{graphicx}
\usepackage{amsthm}

\usepackage{amssymb}
\usepackage{epstopdf}
\usepackage[T1]{fontenc}
\usepackage[ngerman, english]{babel}
\usepackage{mathtools}
\usepackage{bbm}
\usepackage{mathrsfs}
\usepackage{amsmath,amscd}
\usepackage{tikz}
\usepackage{tikz-cd}
\usepackage{babel}
\usetikzlibrary{arrows, matrix}
\usetikzlibrary{cd}
\usepackage[arrow, matrix, curve]{xy}

\newenvironment{proof}{\vspace{1ex}\noindent{\it Proof.}\hspace{0.5em}}
	{\hfill\qed\vspace{1ex}}

\theoremstyle{definition}
\newtheorem{theorem}{Theorem}[section]
\newtheorem{lemma}[theorem]{Lemma}
\newtheorem{proposition}[theorem]{Proposition}
\newtheorem{corollary}[theorem]{Corollary}

\newtheorem{definition}[theorem]{Definition}

\DeclareMathOperator{\Gal}{\operatorname{Gal}}
\DeclareMathOperator{\Q}{\mathbf{Q}}
\DeclareMathOperator{\R}{\mathrm{R}}
\DeclareMathOperator{\Z}{\mathbf{Z}}

\DeclareMathOperator{\A}{\mathbf{A}}

\DeclareMathOperator{\N}{\mathbf{N}}

\DeclareMathOperator{\im}{\mathrm{im}}
\DeclareMathOperator{\Spec}{\operatorname{Spec}}

\DeclareMathOperator{\ord}{\operatorname{ord}}

\DeclareMathOperator{\Lie}{\mathrm{Lie}}

\DeclareMathOperator{\Id}{\mathrm{Id}}
\DeclareMathOperator{\Og}{\mathcal{O}}

\DeclareMathOperator{\Pic}{\mathrm{Pic}}

\DeclareMathOperator{\et}{\acute{\mathrm{e}}{\mathrm{t}}}
\DeclareMathOperator{\Gm}{\mathbf{G}_m}

\DeclareMathOperator{\fppf}{\mathrm{fppf}}

\DeclareMathOperator{\coker}{\mathrm{coker}}
\DeclareMathOperator{\Res}{\mathrm{Res}}

\title{Jumps and motivic invariants of semiabelian Jacobians}
\author{Otto Overkamp}
\date{}
\begin{document}
\maketitle

\begin{abstract} 
We investigate Néron models of Jacobians of singular curves over strictly Henselian discretely valued fields, and their behaviour under tame base change. For a semiabelian variety, this behaviour is governed by a finite sequence of (a priori) real numbers between 0 and 1, called \it jumps. \rm The jumps are conjectured to be rational, which is known in some cases. The purpose of this paper is to prove this conjecture in the case where the semiabelian variety is the Jacobian of a geometrically integral curve with a push-out singularity. Along the way, we prove the conjecture for algebraic tori which are induced along finite separable extensions, and generalize Raynaud's description of the identity component of the Néron model of the Jacobian of a smooth curve (in terms of the Picard functor of a proper, flat, and regular model) to our situation. The main technical result of this paper is that the exact sequence which decomposes the Jacobian of one of our singular curves into its toric and Abelian parts extends to an exact sequence of Néron models. Previously, only split semiabelian varieties were known to have this property.
\end{abstract}

\tableofcontents
\section{Introduction}
Let $K$ be a strictly Henselian discrete valuation field with ring of integers $\Og_K$. Assume that $\Og_K$ is complete with respect to the $\mathfrak{m}$-adic topology, where $\mathfrak{m}\subseteq \Og_K$ denotes the maximal ideal. Denote its residue field by $k$, which we assume to be of characteristic $p>0$. Let $B$ be a semiabelian variety over $K$ with a Chevalley decomposition
$$0\to T\to B\to A\to 0.$$ That is, $T$ is an algebraic torus and $A$ an Abelian variety over $K$. Denote the Néron (lft-)models of $T$, $B$, and $A$ by $\mathscr{T}$, $\mathscr{B}$, and $\mathscr{A}$, respectively. If $L$ is a finite extension of $K$, and $\mathscr{B}_L$ denotes the Néron (lft-)model of $B\times_K\Spec L$, we obtain a canonical morphism
$$\mathscr{B}\times_{\Og_K}\Spec\Og_L\to\mathscr{B}_L.$$ In general, this morphism is not an isomorphism, and the cokernel of the induced map
$$\psi_L\colon \Lie(\mathscr{B})\otimes_{\Og_K}\Og_L\to\Lie(\mathscr{B}_L)$$ contains important information about $B$. (For an introduction to the theory of Lie algebras of group schemes, see \cite{LLR}, Chapter 1). One approach to studying the behaviour of the Néron model of $B$ under tame base change is to consider all tame extensions $K(d)$ of $K$ for positive integers $d$ prime to $p$. For each such $d$, we have
$$\coker\psi_{K(d)}=\bigoplus_{i=1}^g \Og_{K(d)}/\mathfrak{m}_d^{j_{d,i}},$$ with $j_{d,1}\leq j_{d,2}\leq ... \leq j_{d,g},$ where $\mathfrak{m}_d\subseteq \Og_{K(d)}$ denotes the maximal ideal and $g:=\dim B$. Now suppose we are given a tower $K\subseteq K(d_1)\subseteq K(d_2)\subseteq...$ of tame extensions of $K$ which is cofinal in the set of all finite tame extensions of $K$. As it turns out, the limits $$j_i:=\lim_{n\to\infty} \frac{j_{{d_n},i}}{d_n}$$ exist, and are independent of the tower of extensions chosen. The numbers $j_1,..., j_g$ are referred to as the \it jumps \rm of $B$, because of a close link between those numbers and a filtration on the special fibre of $\mathscr{B}$ defined by Edixhoven \cite{E}. For the purposes of this paper, the description above may serve as a definition. Note that, \it a priori, \rm the jumps of $B$ are real numbers. If the semiabelian variety $B$ is tamely ramified (that is, if $B$ acquires semistable reduction after a finite tame extension $K'$ of $K$), then the jumps are rational numbers, we have $[K':K]j_i\in\Z$ for all $i$, and the integer $[K':K]$ is minimal with this property, provided that $K'$ is minimal. On the other hand, if $B$ acquires semistable reduction only after a wild extension $L$ of $K$, much less is known. However, if $B$ is the Jacobian of a smooth proper curve $C$ over $K$ of index one, the jumps still turn out to be rational (see \cite{HN}, Corollary 6.3.1.5). We will begin, in the second chapter, by proving that the jumps of algebraic tori which are induced along a finite separable extension $L/K$ are rational, and we will study other motivic invariants associated to such algebraic tori (see Corollary \ref{inducedjumpscorollary} and Theorem \ref{inducedtorizetatheorem}). Note that we make no further assumptions about the finite extension $L/K$; in particular, we do not assume that this extension is tame. These results are complementary to those in \cite{N}, where the tameness assumption is crucial. Our calculations allow us to answer a question of Halle and Nicaise about the denominators of the jumps of algebraic tori asked in the last chapter of \cite{HN}; see the remark after the proof of Corollary \ref{inducedjumpscorollary}.
We then go on to consider semiabelian varieties over $K$ of the form $\Pic^0_{C/K}$ for a proper, geometrically integral curve $C$ over $K$ which has a \it push-out singularity \rm (see Chapter 5). Such curves are interesting for a number of reasons: Firstly, they form a wide class of singular curves with the property that the toric parts of their Jacobians can be explicitly calculated by hand. Secondly, we shall prove that every curve in this class admits a well-behaved proper model over $\Og_K$ whose construction, beginning from a regular projective model of the curve's normalization, involves only simple operations (blow-ups in regular centres and push-outs). Hence we obtain a wide class of proper non-regular schemes over $\Og_K$ whose Picard functor can nevertheless be explicitly studied.  
The main technical result we shall prove is that if $B$ is the Jacobian of a geometrically integral curve with a push-out singularity over $K$, then the sequence
$$0\to\mathscr{T}\to\mathscr{B}\to\mathscr{A}\to 0$$ of Néron models is exact (Theorem \ref{Neronexacttheorem}). Our method generalizes the approach to Néron models of Jacobians via Picard functors from \cite{BLR}, but requires new ideas to deal with the fact that the models we shall construct are not normal. In order to study the Picard functor of a model of a curve with a push-out singularity, we prove general results about the behavior of push-outs of schemes under base change. Our method may be seen as a partial generalization of the classical description of the Jacobian of a singular curve (in terms of the Jacobian of its normalization and the nature of its singularities) to a relative setting. This provides a new general family of semiabelian varieties whose Chevalley decomposition extends to an exact sequence of Néron models; the only other known class with this property is that of semiabelian varieties with split toric part (compare \cite{CY}, Lemma 11.2). Hence these decomposition results might already be of interest in their own right. Our method of studying the exactness of sequences of Néron models is new, and was therefore not exploited in \cite{HN}. We will see that the toric parts of semiabelian varieties of the form $\Pic^0_{C/K}$ (with $C$ as before) are closely related to induced algebraic tori, and the technical result mentioned above will allow us to use known results about the Jumps (and other motivic invariants) of Jacobians of smooth curves to deduce new theorems in the case of semiabelian Jacobians. Again, there are no tameness assumptions. In the last chapter, we will study motivic Zeta functions associated to Jacobians of curves with push-out singularities. We will show that the motivic zeta function $Z_B(z)$ associated to a semiabelian Jacobian is rational, and that the unique pole of $Z_B(\mathbf{L}^{-s})$ only depends on the Abelian part of $B$ (Theorem \ref{motivicrationaltheorem}). One expects this to be true in general, and our results seem to provide the first class of semiabelian varieties (which are neither tori nor Abelian varieties) beyond the tamely ramified case for which this is known to be true (see the remark at the end of this paper). Throughout the manuscript, we will use the following notation:
\begin{tabbing}
$K$: \hspace{0.2in}\= A strictly Henselian complete discrete valuation field\\
$\Og_F$: \>The ring of integers of a finite extension $F$ of $K$\\
$k$: \>The residue field of $K$\\
$p$: \> The characteristic of $k$\\
$K(d)$:\> The unique finite extension of degree $d$ of $K$ for $p\nmid d$\\
$L$: \> A finite separable extension of $K$\\
$T$:\> The torus $(\Res_{L/K}\Gm)/\Gm$\\
$C$:\> A proper, geometrically integral curve over $K$ with precisely one push-out singularity
\end{tabbing}
\section{Jumps and motivic invariants of certain algebraic tori}
In this section, we will consider the questions already alluded to in the introduction for algebraic tori of the form $\Res_{L/K}\Gm$, and $T:=(\Res_{L/K}\Gm)/\Gm$, where $L/K$ is a finite separable extension. These tori are of particular interest to us because toric parts of semiabelian Jacobians are very closely related to such tori, as we shall see later.  We will, however, introduce some concepts in this chapter in slightly greater generality for later use. Throughout this paper, when referring to \it Néron lft-modles \rm we mean models which satisfy the Néron mapping property but which are only required to be locally of finite type, rather than of finite type, as is the case with classical Néron models. See \cite{BLR}, Chapter 10, for more details. 
\begin{definition} \label{elementarydefinition}
Let $B$ be a semibelian variety over $K$. For each positive integer $d$ prime to $p$, denote by $\mathscr{B}(d)$ the Néron lft-model of $B\times_K\Spec K(d)$ over $\Og_{K(d)}.$ We have a canonical morphism
$$\psi_{d}\colon \Lie(\mathscr{B})\otimes_{\Og_K}\Og_{K(d)}\to \Lie \mathscr{B}(d).$$ We can find nonnegative integers $j_{d,1}\leq ... \leq j_{d,g}$, where $g:=\dim B$, such that $$\coker\psi_{d}\cong \bigoplus_{i=1}^g \Og_{K(d)} /\mathfrak{m}_d^{j_{d,i}}.$$\\
(i) The numbers $j_{d,1}\leq ... \leq j_{d,g}$ are called the \it $d$-jumps \rm of $B$.\\
(ii) Choose a tower $K\subseteq K(d_1)\subseteq K(d_2)\subseteq...$ of finite tame extensions of $K$ which is cofinal in the set of all finite tame extensions of $K$. For each $i=1,...,g$, let 
$$j_i:=\lim_{n\to\infty}\frac{j_{{d_n},i}}{d_n}.$$ This limit exists by \cite{HN}, (6.1.3.7). The numbers $j_1\leq ... \leq j_g$ are called the \it jumps \rm of $B$.\\
(iii) For each $d$ prime to $p$, define
$$\ord_B(d):=\sum_{i=1}^g j_{d,i}=\mathrm{length}_{\Og_K(d)}(\coker\psi_{d}).$$ The resulting function is called the \it order function of \rm $B$.\\
(iv) Define the \it tame base change conductor \rm $c_{tame}(B)$ of $B$ to be 
$$c_{tame}(B):=\sum_{i=1}^g j_i.$$
\end{definition}
In order to state the next definition, recall the Grothendieck ring of varieties over $k,$ denoted by $K_0(\mathrm{Var}_k)$. For a (reduced) algebraic variety $X$ over $k$, denote by $[X]$ the class of $X$ in $K_0(\mathrm{Var}_k)$. For more details about this ring, see \cite{Looi}, p. 269, or \cite{NS}. As usual, put
$$\mathbf{L}:=[\mathbf{A}^1_k]\in K_0(\mathrm{Var}_k).$$ Also recall that, for each $d$ prime to $p$, the Néron lft-model $\mathscr{B}(d)$ has a maximal quasi-compact open subgroup scheme, denoted by $\mathscr{B}(d)^{\mathrm{qc}}.$
\begin{definition} Define the \it motivic Zeta function of \rm $B$ as
$$Z_B(z):=\sum_{p\nmid d} [\mathscr{B}(d)_k^{\mathrm{qc}}]\mathbf{L}^{\ord_B(d)}z^d\in K_0(\mathrm{Var_k})[\![z]\!].$$
\end{definition}
One of the objectives of this paper is to prove that the motivic Zeta function of a semiabelian Jacobian is a rational function. We follow the strategy already set out in \cite{HN2}. For this it is necessary to understand the behaviour of the order function, and of the classes $[\mathscr{B}(d)_k^{\mathrm{qc}}]$, for $d$ prime to $p$. We shall do this in this chapter for algebraic tori which are induced along a finite separable extension. In particular, it will follow that the jumps of induced algebraic tori are rational numbers. Similar results are only known in the tamely ramified case, and beyond the cases of tamely ramified or induced tori the question of rationality of jumps or motivic Zeta functions remains wide open. 
\subsection{Jumps and Galois actions}
Let $B$ be a semiabelian variety over $K$ with Néron lft-model $\mathscr{B}$ over $\Og_K$. For each $d$ prime to $p$, put $B(d):=B\times_K\Spec K(d),$ and let $\mathscr{B}(d)$ be the Néron lft-model of $B(d)$. For all $d$ prime to $p$, let $\boldsymbol{\mu}_d$ denote the group of $d$-th roots of unity in $\Og_K$. Since $\Og_K$ is strictly Henselian, there is a natural isomorphism $\boldsymbol{\mu}_d\cong \Gal(K(d)/K)$, and we obtain a right action of $\boldsymbol{\mu}_d$ on $B(d)$, and hence on $\mathscr{B}(d)$, which is compatible with the right action of $\boldsymbol{\mu}_d$ on $\Spec\Og_{K(d)}.$ Now, the right action of $\boldsymbol{\mu}_d$ on $\mathscr{B}(d)$ induces a right action of $\boldsymbol{\mu}_d$ on the Abelian group $\Lie \mathscr{B}(d)$. This is not in general an action on the $\Og_{K(d)}$-module $\Lie \mathscr{B}(d)$. However, the right $\boldsymbol{\mu}_d$-action on $\Lie\mathscr{B}(d)$ and the left $\boldsymbol{\mu}_d$-action on $\Og_{K(d)}$ are compatible in the sense that for $\lambda\in \Og_{K(d)}$, $x\in \Lie\mathscr{B}(d)$, and $\zeta\in \boldsymbol{\mu}_d$, we have
$$(\lambda x)\ast\zeta=(\zeta^{-1}\cdot \lambda)(x\ast\zeta).$$ Since the maximal ideal $\mathfrak{m}_d\subseteq \Og_{K(d)}$ is invariant under the left action of $\boldsymbol{\mu}_d$, the right action of $\boldsymbol{\mu}_d$ on $\Lie\mathscr{B}(d)$ induces a right action on the $k$-vector space $(\Lie\mathscr{B}(d))\otimes_{\Og_{K(d)}} k$ (i.e. this action is $k$-linear). In what follows, we will prove two lemmata which allow us to use this Galois action to deduce new information about the jumps of $B$. Let us begin with the following well-known
\begin{lemma}
Keep the notation from the beginning of this subsection, and let $j_{d,1},..., j_{d,g}$ ($g=\dim B$) be the $d$-jumps of $B$. Furthermore, denote by $\chi_d$ the one-dimensional right $\boldsymbol{\mu}_d$-representation over $k$ defined by $y\ast \zeta:=\zeta y$ for $y\in k$. Then $j_{d,i}\leq d-1$ for all $1\leq i\leq g$, and there is an isomorphism
$$(\Lie\mathscr{B}(d))\otimes_{\Og_K(d)} k \cong \bigoplus_{i=1}^d \chi_d^{\otimes j_{d,i}}$$ of right $k[\boldsymbol{\mu}_d]$-modules. In particular, the $d$-jumps of $B$ only depend upon the $k[\boldsymbol{\mu}_d]$-module $(\Lie\mathscr{B}(d))\otimes_{\Og_{K(d)}}~k.$
\end{lemma}
\begin{proof}
This follows from \cite{HN2}, Corollary 4.8 and Theorem 4.10. See also \cite{HN}, 6.1.3.4.
\end{proof}\\
We shall also need
\begin{proposition}
Let $0\to T\to B\to A\to 0$ be an exact sequence of semiabelian varieties over $K$. Suppose that, for some $d$ prime to $p,$ the sequence of Néron lft-models
$$0\to \mathscr{T}(d)\to \mathscr{B}(d)\to \mathscr{A}(d)\to 0$$ is exact. Let $\mathcal{I}_B^d$ be the multiset of $d$-jumps of $B$, and similarly for $T$ and $A$. Then we have
$$\mathcal{I}_B^d=\mathcal{I}_T^d\cup \mathcal{I}^d_A$$ as multisets. \label{djumpsproposition}
\end{proposition}
\begin{proof}
First note that the sequence 
$$0\to \mathscr{T}(d)\to \mathscr{B}(d)\to \mathscr{A}(d)\to 0$$ is $\boldsymbol{\mu}_d$-equivariant. This follows immediately from the universal property of the Néron lft-model. Consider the sequence
$$0\to \Lie\mathscr{T}(d)\to \Lie\mathscr{B}(d)\to \Lie\mathscr{A}(d)\to 0$$ of $\Og_{K(d)}$-modules, which is exact by \cite{LLR}, Proposition 1.1(a),(c).  By our previous observation, each morphism in this sequence respects the right $\boldsymbol{\mu}_d$-action. Since $\Lie\mathscr{A}(d)$ is free over $\Og_{K(d)}$, this sequence is split exact, which implies that the sequence
$$0\to (\Lie \mathscr{T}(d))\otimes_{\Og_{K(d)}} k\to (\Lie \mathscr{B}(d))\otimes_{\Og_{K(d)}} k \to (\Lie \mathscr{A}(d))\otimes_{\Og_{K(d)}} k\to 0$$ of $k$-vector spaces is exact as well. A simple check shows that $\boldsymbol{\mu}_d$-equivariance is not affected by taking the tensor products, which implies that the last sequence is, in fact, an exact sequence of right $k[\boldsymbol{\mu}_d]$-modules. However, since $p\nmid d$, the category $k[\boldsymbol{\mu}_d]-\mathbf{Mod}$ is semisimple, which implies that there is a $\boldsymbol{\mu}_d$-equivariant isomorphism
$$(\Lie \mathscr{B}(d))\otimes_{\Og_{K(d)}} k\cong (\Lie \mathscr{T}(d))\otimes_{\Og_{K(d)}} k\oplus (\Lie \mathscr{A}(d))\otimes_{\Og_{K(d)}} k.$$ This implies the claim. 
\end{proof}
\subsection{Jumps of induced tori}
Let $L/K$ be a finite separable extension. For each $d$ prime to $p$, let $\mathscr{R}(d)$ be the Néron lft-model of $(\Res_{L/K}\Gm)\times_{K}\Spec K(d) =\Res_{L\otimes_KK(d)/K(d)}\Gm.$ Then we have a canonical isomorphism 
$$\Lie(\mathscr{R}(d))\cong \Og_{L\otimes_KK(d)}.$$
This follows from the construction of the Néron lft-model of $\Gm$ in \cite{BLR}, Chapter 10.1, Example 5, and the fact that Néron lft-models and Lie algebras commute with Weil restriction. In general, $L\otimes_KK(d)$ will only be a finite étale algebra over $K(d)$, so $L\otimes_KK(d)=L_1\times...\times L_r,$ where the $L_j$ are finite separable extensions of $K(d)$. In this case, we have $\Res_{L\otimes_KK(d)/K(d)}\Gm=\Res_{L_1/K(d)}\Gm\times...\times \Res_{L_r/K(d)}\Gm$, and we use the notation $\Og_{L\otimes_KK(d)}=\Og_{L_1}\times...\times \Og_{L_r}.$ In particular, if we put $\ord(-):=\ord_{\Res_{L/K}\Gm}(-),$ we have
$$\ord(d)=\mathrm{length}_{\Og_{K(d)}}(\coker(\Og_L\otimes_{\Og_K}\Og_{K(d)}\to \Og_{L\otimes_KK(d)})).$$
\begin{lemma}
Suppose that $d\equiv 1\mod [L:K].$ Then there is an isomorphism
$$\coker(\Og_L\otimes_{\Og_K}\Og_{K(d)}\to \Og_{L\otimes_KK(d)})\cong \bigoplus_{\nu=0}^{[L:K]-1}\Og_{K(d)}/\mathfrak{m}_d^{\nu\frac{d-1}{[L:K]}}.$$  \label{inducedjumpslemma}
\end{lemma}
\begin{proof}
Let $P(t)\in \Og_{K}[t]$ be an Eisenstein polynomial such that $\Og_L\cong \Og_K[t]/\langle P(t)\rangle.$ This is possible because the extension $L/K$ is totally ramified. Since $d\equiv 1\mod [L:K],$ we can find an Eisenstein polynomial $Q(t)\in \Og_{K(d)}[t]$ such that 
$$P(\pi_d^{\frac{d-1}{[L:K]}}t)=\pi_d^{d-1}Q(t),$$ where $\pi_d$ denotes some uniformizing element of $K(d)$. In particular, we obtain an isomorphism
$$\Og_{L\otimes_KK(d)}\cong \Og_{K(d)}[t]/\langle Q(t)\rangle.$$ Under this isomorphism, the sequence $1,t,..., t^{[L:K]-1}$ is an $\Og_{K(d)}$-basis of $\Og_{L\otimes_KK(d)}$, whereas the sequence $1, \pi_d^{\frac{d-1}{[L:K]}}t, ..., \pi_d^{\frac{([L:K]-1)(d-1)}{[L:K]}}t^{[L:K]-1}$ is an $\Og_K(d)$-basis of the image of the morphism
$$\Og_L\otimes_{\Og_K}\Og_{K(d)}\to \Og_{L\otimes_KK(d)}.$$ This implies the claim.
\end{proof}
\begin{corollary}
Let $L/K$ be a finite separable extension. Then the jumps of the torus $\Res_{L/K}\Gm$ are $$0, \frac{1}{[L:K]},...,\frac{[L:K]-1}{[L:K]}.$$ In particular, we have \label{inducedjumpscorollary}
$$c_{tame}(\Res_{L/K}\Gm)=\frac{[L:K]-1}{2}.$$
\end{corollary}
\begin{proof}
Let $\delta\in \N$ be such that $K(\delta)\subseteq L$ and such that the extension $L/K(\delta)$ is purely wild. First note that, in order to calculate the jumps, we may also consider a tower of finite tame extensions $K\subseteq K(d_1)\subseteq K(d_2)\subseteq ...$ which has the property that all $d_n$ are prime to $\delta$, and such that the tower is cofinal in the set of all finite tame extensions of $K$ whose degree is prime to $\delta$. This follows from the description of the jumps in terms of Edixhoven's filtration (see \cite{HN}, Chapter 6, or Proposition \ref{filtrationproposition} below) and the fact that $\Z_{\langle \delta\rangle}\cap\Z_{\langle p\rangle}\cap [0,1[$ is dense in $\Z_{\langle p\rangle}\cap [0,1[.$ In this case, we may assume that $d_n\equiv 1\mod[L:K]$ for all $n\in\N$. Since $d_n\to \infty $ as $n\to \infty$, we find that
$$\lim_{n\to \infty} \frac{\nu (d_n-1)}{d_n[L:K]}=\frac{\nu}{[L:K]}$$ for $\nu=0,..., [L:K]-1$. Hence the previous lemma implies the result. 
\end{proof}\\
\\
\noindent$\mathbf{Remark}.$ The corollary above shows that the lowest common multiple of the denominators of the jumps of a torus $T$ need not, in general, coincide with the degree $[F:K]$, where $F$ is the splitting field of $T$. Indeed, suppose that $T=\Res_{L/K}\Gm$ for some finite separable extension $L/K$ which is not Galois. Then the lowest common multiple of the denominators of the jumps is given by $[L:K]$, whereas the splitting degree of $T$ is strictly greater than $[L:K].$ This answers the second part of Question 10.1.2 of \cite{HN} negatively. 
\subsection{An example of non-invariance under isogeny}
Let $L/K$ be a Galois extension with Galois group isomorphic to $\Z/2\Z\times\Z/2\Z$. For example, one could choose an algebraic closure $\overline{\mathbf{F}}_2$ of $\mathbf{F}_2$ and let $K$ be the field of fractions of $W(\overline{\mathbf{F}}_2),$ where $W(-)$ stands for the ring of Witt vectors. If $\zeta_8$ denotes a primitive $8$th root of unity in $\overline{K}$, the Galois group of $L:=K(\zeta_8)$ over $K$ is canonically isomorphic to $(\Z/8\Z)^\times.$ Choose a presentation
$$\Gal(L/K)= \langle \sigma, \tau\mid \sigma^2=\tau^2=e, \sigma\tau=\tau\sigma\rangle.$$ For each element $\alpha$ of this group, denote by $L_\alpha$ the subfield of $L$ fixed by $\alpha.$ We will now show that there is an isogeny $$\Res_{L/K}\Gm\to \Gm\times\Res^1_{L_\sigma/K}\Gm\times\Res^1_{L_\tau/K}\Gm\times\Res^1_{L_{\sigma\tau}/K}\Gm.$$  (If $F/K$ is a finite separable extension, $\Res^1_{F/K} \Gm$ denotes the \it norm one torus \rm associated to this extension, which is defined to be the kernel of the norm map $\Res_{F/K}\Gm \to \Gm.$ If $F/K$ is quadratic, then $X^\ast(\Res^1_{F/K}\Gm)$ is isomorphic to $\Z$ as a $\Z$-module, and the unique generator of $\Gal(F/K)$ acts as multiplication by $-1.$)
To see this, note that there is a canonical isomorphism
$$X^\ast(\Res_{L/K}\Gm)\cong \Z[\Gal(L/K)].$$
Now denote by $V_\sigma$ the $\Z[\Gal(L/K)]$-module which is free of rank one and on which $\sigma$ acts as multiplication by $-1$, and $\tau$ as the identity. Define $V_{\tau}$ analogously, and let $V_{\sigma\tau}$ be the $\Z[\Gal(L/K)]$-module which is free of rank one and on which both $\sigma$ and $\tau$ act as multiplication by $-1$. Consider the morphism of $\Z[\Gal(L/K)]$-modules
$$\Z\oplus V_{\sigma}\oplus V_{\tau}\oplus V_{\sigma\tau}\to \Z[\Gal(L/K)]$$ given by
\begin{align*}
(1,0,0,0)&\mapsto e+\sigma+\tau+\sigma\tau\\
(0,1,0,0)&\mapsto e-\sigma+\tau-\sigma\tau\\
(0,0,1,0)&\mapsto e+\sigma-\tau-\sigma\tau\\
(0,0,0,1)&\mapsto e-\sigma-\tau+\sigma\tau.
\end{align*}
A direct calculation shows that this morphism has determinant equal to $16$, so it becomes an isomorphism after tensoring with $\Q$. Therefore, the map of $\Z[\Gal(L/K)]$-modules we defined does indeed give rise to an isogeny as above. It follows that the jumps of algebraic tori are not invariant under isogeny: The jumps of the torus on the left are $0,1/4,1/2,3/4$, whereas the jumps of the torus on the right are $0, 1/2, 1/2, 1/2.$ This is in contrast with the situation in the tamely ramified case, where the jumps are invariant under isogeny. 
\subsection{The order function of induced algebraic tori}
In this section, we will consider the order functions associated to induced algebraic tori. Let $L/K$ be a finite separable extension of degree $n$, let $\mathscr{R}$ denote the Néron model of the torus $\Res_{L/K}\Gm,$ and for each $d$ prime to $p$, denote by $\mathscr{R}(d)$ the Néron model of $(\Res_{L/K}\Gm)\times_KK(d) \cong \Res_{L\otimes_KK(d)/K(d)}\Gm$ over $\Og_{K(d)}.$
In order to study the order function of induced algebraic tori in this generality, it seems necessary to use the fact that the jumps of a semiabelian variety as defined in the introduction coincide with the jumps of a filtration on the special fibre of $\mathscr{R}$ defined by Edixhoven \cite{E}. For a more comprehensive treatment of this filtration, see \cite{E} or \cite{HN2}, Chapter 4. We shall need the following results, which we state only in the case of induced tori, but which are valid in much greater generality:
\begin{proposition}
For each $d$ prime to $p$, there is a decreasing filtration $F_d^a\mathscr{R}_k$ on $\mathscr{R}_k$, where the index $a$ lies in $\{ 0,..., d\},$ with the following property: For each $0\leq j \leq d-1,$ let $\mathrm{Gr}_d^ {j} \mathscr{R}_k:= F_d^{j}\mathscr{G}_k/F_d^{j +1}\mathscr{G}_k.$ If $j_{d,1},..., j_{n,1}$ denote the $d$-jumps of $\Res_{L/K}\Gm$, then 
$$\dim \mathrm{Gr}^{j}_d \mathscr{R}_k=\#\{i=1,..., n\colon j_{d,i}=j\}.$$
Furthermore, for all $0\leq j\leq d-1$ and all $d$, $m$ prime to $p$, we have 
$$F^{jm}_{dm}\mathscr{R}_k=F^{j}_d\mathscr{R}_k,$$ $F_d^0\mathscr{R}_k=\mathscr{R}_k,$ and $F^d_d\mathscr{R}_k=0.$
\end{proposition}
\begin{proposition}
There is a decreasing filtration $\tilde{F}^\alpha \mathscr{R}_k$ on $\mathscr{R}_k$, where the index $\alpha$ lies in $\Z_{\langle p\rangle}\cap [0,1[,$ such that for any $d$ prime to $p$ and $i=0,..., d-1$, we have 
$$\tilde{F}^{i/d}\mathscr{G}_k = F_d^i\mathscr{G}_k.$$  It has the following property: For a real number $0\leq j<1$, define \label{filtrationproposition}
$$\tilde{F}^{>j}\mathscr{R}_k:=\tilde{F}^\beta\mathscr{R}_k,$$ where $\beta\in \Z_{\langle p\rangle}\cap ]j,1[$ is chosen such that $\tilde{F}^{\beta'}\mathscr{R}_k=\tilde{F}^{\beta}\mathscr{R}_k$ for all $\beta'\in \Z_{\langle p\rangle}\cap ]j,\beta]$. If $j\not=0$, define $$\tilde{F}^{<j}\mathscr{R}_k:=\tilde{F}^\gamma\mathscr{R}_k,$$ where $\gamma\in \Z_{\langle p\rangle}\cap [0,j[$ has been chosen such that $\tilde{F}^{\gamma'}\mathscr{R}_k=\tilde{F}^{\gamma}\mathscr{R}_k$ for all $\gamma'\in \Z_{\langle p \rangle}\cap [\gamma, j[.$ Put $\tilde{F}^{<0}\mathscr{R}_k:=\mathscr{R}_k$. Then define 
$$\tilde{\mathrm{Gr}}^j\mathscr{R}_k:=\tilde{F}^{<j}\mathscr{R}_k/\tilde{F}^{>j}\mathscr{R}_k.$$
If the jumps of $\Res_{L/K}\Gm$ are $j_1\leq...\leq j_n$, then for all $j\in [0,1[,$ we have
$$\dim \tilde{\mathrm{Gr}}^j\mathscr{R}_k=\#\{i=1,...,n\colon j_i=j\}.$$
\end{proposition}
\begin{proof} See \cite{HN2}, Chapter 4, for both preceding propositions. 
\end{proof}
\begin{corollary}
Let $L/K$ be a finite separable extension and let $d$ be prime to $p$. Then the $d$-jumps of $\Res_{L/K}\Gm$ are given by $$0, \bigg\lfloor \frac{d}{[L:K]} \bigg\rfloor,..., \bigg\lfloor \frac{([L:K]-1)d}{[L:K]}\bigg\rfloor.$$
\end{corollary}
\begin{proof}(Compare \cite{HN2}, Proposition 7.5)
Let $d$ be prime to $p$, and let $i\in \{0,..., d-1\}.$ By definition, the multiplicity of $i$ as a $d$-jump is given by
$$\dim F^i_d\mathscr{R}_k-\dim F^{i+1}_d\mathscr{R}_k = \dim \tilde{F}^{\frac{i}{d}}\mathscr{R}_k-\dim\tilde{F}^{\frac{i+1}{d}}\mathscr{R}_k.$$ Hence, the multiplicity of $i$ as a $d$-jump is given by the sum of the multiplicities of the jumps contained in the half-open interval $[i/d, (i+1)/d[.$ However, by Corollary \ref{inducedjumpscorollary}, the jumps of $\Res_{L/K}\Gm$ are given by $0, 1/[L:K],..., ([L:K]-1)/[L:K].$ Hence, all one has to check is that for $i=0,..., d-1,$
$$\#\Bigg(\bigg[\frac{i}{d}, \frac{i+1}{d}\bigg[\cap \bigg\{0, \frac{1}{[L:K]},..., \frac{[L:K]-1}{[L:K]}\bigg\}\Bigg) =\#\{\lambda=0,..., [L:K]-1\colon \bigg \lfloor \frac{d\lambda}{[L:K]}\bigg\rfloor=i\}.$$ 
That, however, is straightforward.  
\end{proof}

\begin{proposition}
As before, put $\ord(d):=\ord_{\Res_{L/K}\Gm}(d)$ for all $d$ prime to $p$. Let $L/K$ be a finite separable extension. Let $\alpha\in \N$ be prime to $p$ and let $q$ be a nonnegative integer such that $\alpha+q[L:K]$ is prime to $p$. Then we have \label{torusorderfunctionproposition}
$$\ord(\alpha+q[L:K])=\ord(\alpha)+q[L:K]c_{tame}(\Res_{L/K}\Gm).$$
\end{proposition}
\begin{proof}
By our previous results, we have
$$\ord(d)=\sum_{i=0}^{[L:K]-1} \bigg\lfloor \frac{di}{[L:K]}\bigg\rfloor.$$ For each $i\in \{0,..., [L:K]-1\},$ we have
$$\bigg\lfloor \frac{(d+q[L:K])i}{[L:K]}\bigg\rfloor=\bigg\lfloor \frac{di}{[L:K]}\bigg\rfloor+qi.$$ This implies the Proposition.
\end{proof}\\
\\

As a first result, let us record
\begin{theorem} Let $L/K$ be a finite, separable, and purely wild extension. Then the motivic Zeta function $Z_{\Res_{L/K}\Gm}(z)$ is a rational function. More precisely, it is contained in the subring
$$K_0(\mathrm{Var}_k)\Big[z, \frac{1}{1-\mathbf{L}^az^b}\Big]_{(a,b)\in \Z\times\Z_{>0}\colon a/b=c_{tame}(\Res_{L/K}\Gm)}\subseteq K_0(\mathrm{Var}_k)[\![z]\!].$$ The function $Z_{\Res_{L/K}\Gm}(\mathbf{L}^{-s})$ has a unique pole at $s=c_{tame}(\Res_{L/K}\Gm)$ of order 1. \label{inducedtorizetatheorem}
\end{theorem}
\begin{proof}
Let $\mathscr{R}(d)$ denote the Néron lft-model of $\Res_{L(d)/K(d)}\Gm.$ Then we have $\mathscr{R}(d)^{\mathrm{qc}}=\mathscr{R}(d)^0\cong \Res_{\Og_{L(d)}/\Og_{K(d)}}\Gm$ for all $d$ not divisible by $p$. Using the previous Lemma, we can take the proof of \cite{HN}, Theorem 8.3.1.2 \it mutatis mutandis \rm, provided we can prove that $[\mathscr{R}(\alpha)^0_k]=[\mathscr{R}(1)^0_k]$ in $K_0(\mathrm{Var}_k)$ for all $\alpha$ prime to $p$. However, this is clear, since $\mathscr{R}(d)^0_k\cong \mathscr{R}(1)^0_k$ as varieties. 
\end{proof}\\
\\
\noindent $\mathbf{Example.}$ Suppose $\mathrm{char} \,k=2,$ and let $L/K$ be a quadratic extension of $K$. In this case, we can calculate $Z_{\Res_{L/K}\Gm}(z)$ explicitly. We have
\begin{align*}
Z_{\Res_{L/K}\Gm}(z)&=\sum_{2\nmid d} [(\Res_{\Og_L(d)/\Og_K(d)}\Gm)\times_{\Og_K(d)} k] \mathbf{L}^{\ord(d)}z^d\\
&=\sum_{2\nmid d} (\mathbf{L}-1)\mathbf{L}^{(d+1)/2}z^d\\
&=\sum_{q=0}^\infty (\mathbf{L}-1)\mathbf{L}^{1+q}z^{1+2q}\\
&=(\mathbf{L}-1)\mathbf{L}z\sum_q (\mathbf{L}z^2)^q\\
&=\frac{(\mathbf{L}-1)\mathbf{L}z}{1-\mathbf{L}z^2}.
\end{align*}
For the second equation we have used that the special fibre of $\Res_{\Og_L(d)/\Og_K(d)}\Gm$ is equal to $\Res_{(k[\epsilon]/\langle \epsilon^2 \rangle) /k}\Gm=\Gm\times_k\mathbf{G}_{\mathrm{a}},$ and that $\ord(d):=\ord_{\Res_{L/K}\Gm}(d)=\frac{d-1}{2}$ for all $d$ prime to 2 by Lemma \ref{inducedjumpslemma}.
\section{The Néron model of the torus $T$}
In this section, we will explicitly construct the Néron model of the torus $$T=(\Res_{L/K}\Gm)/\Gm,$$ where $L/K$ is a finite separable extension. Those tori are of particular interest to us because they arise as toric parts of semiabelian Jacobians. Choose a uniformizer $\pi_{L}$ of $L$, and a unit $c\in \Og_{L}^\times$ such that $\pi_{L}^{[L:K]}=c\pi_K.$ For each $\alpha\in  \Z$ put
$$X_\alpha:=(\Res_{\Og_L/\Og_K}\Gm)/\Gm.$$ Furthermore, let $U_\alpha\subseteq X_\alpha$ be the generic fibre.
Define gluing data as follows: Consider the isomorphism 
$$\phi_{\alpha}\colon U_{\alpha}\to U_{\alpha+1}$$ defined by multiplication by the element $\pi_L$. Also define the isomorphism
$$\psi_{\alpha}\colon X_{\alpha}\to X_{\alpha+[L:K]}$$ defined by multiplication by the element $c$. Compatibility of these isomorphisms is clear, so we do indeed obtain gluing data. Define the group scheme $\mathscr{T}$ to be the resulting scheme together with the evident group structure. This is indeed a Néron model of $T$: First observe that $\mathscr{T}$ is of finite type over $\Og_K$, and that it is a model of $T.$ Hence we can apply \cite{BLR}, Chapter 7.1, Theorem 1, and all we have to show is that the canonical map
$$\mathscr{T}(\Og_K)\to T(K)$$ is bijective. However, this claim follows directly from the construction.  From this, we may deduce the following result about the structure of $\mathscr{T}$: 
\begin{corollary}
Let $L/K$ be a finite separable extension and let $\mathscr{T}$ be the Néron model of the torus $T$ associated to $L$ as above. Let $\mathscr{T}^0$ be the identity component of $\mathscr{T}.$ Then we have\\
(i) the canonical map
$$(\Res_{\Og_L/\Og_K}\Gm)/\Gm\to \mathscr{T}^0$$ is an isomorphism, and\\
(ii) there is an exact sequence \label{torusneroncorollary}
$$0\to \mathscr{T}^0_k\to \mathscr{T}_k\to \Z/[L:K]\Z\to 0.$$
\end{corollary}
\begin{proof}
The first statement follows directly from the construction of $\mathscr{T}$ preceding this Corollary. For the second statement, consider the commutative diagram
$$\begin{CD}
0@>>>\Og_K^\times@>>>K^\times@>>>\Z @>>> 0\\
&&@VVV@VVV@VV{\cdot [L:K]}V\\
0@>>> \Og_L^\times @>>> L^\times @>>> \Z @>>> 0.
\end{CD}$$
The cokernels of the first two vertical arrows are $\mathscr{T}^0(\Og_K)$ and $\mathscr{T}(\Og_K),$ respectively. It follows from Hensel's lemma (\cite{JM}, Lemma 1.5) that the morphism $\mathscr{T}(\Og_K)\to \mathscr{T}_k(k)/\mathscr{T}_k^0(k)=\mathscr{T}_k/\mathscr{T}_k^0$ is surjective, and one sees easily that its kernel is precisely $\mathscr{T}^0(\Og_K)$. Now one can use the snake lemma to deduce the claim.
\end{proof}
\section{Preliminaries about push-outs and cohomological flatness in dimension zero}
In the following sections, we will need the notion of \it push-out of schemes. \rm We will always use the categorical definition of push-outs; see \cite{Kash}, p. 46, where the terminology \it fiber coproduct \rm is used. It is well-known that such push-outs do not always exist: For example, it is impossible to contract a line to a point on $\mathbf{P}^2_k$ for any field $k$. However, Schwede \cite{Schw} has shown that push-outs do exist in certain circumstances (see also \cite{Fe}). Throughout this section, we work over an affine base scheme $S=\Spec R,$ where we assume $R$ to be Noetherian. First of all, we have
\begin{lemma} (Schwede) Let $X$ be an affine scheme and let $Y$ be a closed subscheme of $X$. Suppose that there is a morphism $Y\to Z$ to some affine scheme $Z$. Then the push-out $X\cup_Y Z$ (taken in the category of ringed spaces) is an affine scheme. If $A$ and $B$ are rings such that $X=\Spec A,$ $Z=\Spec B$ and $I\subseteq A$ is the ideal defining $Y$, then \label{pushoutexistencelemma}
$$X\cup_YZ\cong\Spec( A\times_{A/I} B).$$ Furthermore, the morphisms $X\to X\cup_YZ$ and $Z\to X\cup_YZ$ are morphisms of schemes, and they turn $X\cup_YZ$ into a push-out in the category of schemes. The induced morphism $X\backslash Y\to (X\cup_YZ)\backslash Z$ is an isomorphism of schemes. 
\end{lemma}
\begin{proof}
See \cite{Schw}, Theorems 3.4 and 3.5.
\end{proof}\\
If the schemes $X$ and $Z$ are $S$-schemes, then the push-out $X\cup_YZ$ comes with a natural $S$-scheme structure, and $X\cup_YZ$ is the push-out in the category of $S$-schemes. Even if push-outs exist, they can by quite pathological: For example, one can contract a line on $\mathbf{A}^2_k$ to a point over any field $k$, but the resulting scheme is not of finite type over $k$ (see \cite{Schw}, Example 3.7). We will now show that, on the other hand, if one restricts to contracting closed subschemes which are finite over the base, the result will still be of finite type:
\begin{proposition}
Let $X$ be a scheme of finite type over $S$. Let $Y$ be a closed subscheme of $X$ which is contained in an open affine subscheme of $X$, and assume further that $Y$ is finite and faithfully flat over $S$. Then the push-out $X\cup_YS$, taken in the category of ringed spaces, exists and is a scheme of finite type over $S$. Furthermore, the maps $X\to X\cup_YS$ and $S\to X\cup_YS$ turn $X\cup_YS$ into a push-out in the category of schemes. If $X$ is proper over $S$ then so is $X\cup_YS$. \label{pushoutfiniteproposition}
\end{proposition}
\begin{proof}
Let us first prove the first two statements. To see that the push-out exists, we may cover $X$ with open affine subschemes, with one containing $Y$, and the other ones intersecting $Y$ trivially. Then we may consider the push-out of our affine neighbourhood of $Y$ along the morphism $Y\to S.$ Since the other affine open subschemes remain unaffected, we can glue them together to obtain the desired push-out, using the previous Theorem. To prove that the resulting scheme is of finite type over $S$, one immediately reduces to the case where $X=\Spec A$ for some $R=\Gamma(S, \Og_S)$-algebra $A$ of finite type. Let $I\subseteq A$ denote the ideal defining $Y$. Let us show that $A$ is finitely generated as a module over $A\times_{A/I} R.$ Let $f\in A$ and pick elements $\alpha_1,..., \alpha_m$ of $A$ such that their images generate $A/I$ as an $R$-module. Then there exist elements $r_1,..., r_m\in R$ such that $f-\sum_i r_i\alpha_i\in I.$ But both $R$ and $I$ clearly lie in the image of the canonical map $A\times_{A/I}R\to A$. Therefore, the set $\{1,\alpha_1,...,\alpha_m\}$ generates $A$ as a module over $A\times_{A/I} R.$ Note also that the map $A\times_{A/I}R\to A$ is injective since the morphism $R\to A/I$ is faithfully flat, and hence injective, by assumption. Since $R$ is Noetherian by assumption, we may apply the Lemma of Artin and Tate (\cite{AM}, Proposition 7.8) to deduce that $A\times_{A/I}R$ is finitely generated as an algebra over $R$. Hence the first two statements follow.
Now suppose that $X$ is proper over $S$. Since the morphism $X\to X\cup_YS$ is surjective (this follows because the morphism $X\to X\cup_Y S$ is scheme-theoretically dominant and finite), and since surjectivity is stable under base change, it follows that the morphism $X\cup_YS\to S$ is universally closed. We already know from the first part that $X\cup_YS$ is of finite type over $S$, so all that remains to be shown is that $X\cup_YS$ is separated over $S$. To show separatedness, we use the valuative criterion. Let $\mathcal{R}$ be a discrete valuation ring over $S$ with field of fractions $\mathcal{F}$. Suppose we have two $S$-morphisms $\varphi_i\colon\Spec\mathcal{R} \to X\cup_YS$ ($i=1,2$) which restrict to the same morphism $\Spec\mathcal{F}\to X\cup_YS.$ Suppose first that this morphism factors through the closed immersion $S\to X\cup_YS.$ Then so do the $\varphi_i$, so we immediately obtain $\varphi_1=\varphi_2$. On the other hand, if $\Spec \mathcal{F}\to X\cup _YS$ does not factor through $S\to X\cup_YS$, then there is a unique morphism $\varphi \colon \Spec\mathcal{F}\to X$ lifting this morphism. However, since $X\to X\cup_YS$ is finite, it is in particular proper, so we can lift the morphisms $\varphi_i$ to morphisms $\phi_i\colon \Spec\mathcal{R}\to X$ extending $\varphi.$ Since $X$ is proper over $S$, it is also separated over $S$, which implies that the two lifts $\phi_i$ coincide. Hence $\varphi_1=\varphi_2.$
\end{proof}
\begin{lemma}
Now let $S=\Spec R$, and suppose that $A$ is an $R$-algebra. Suppose the ideal $I\subseteq A$ is such that $A/I$ and $(A/I)/R$ are both finitely generated flat $R$-modules and such that the map $R\to A/I$ is injective. Then forming $A\times_{A/I} R$ commutes with arbitrary base change. That is, for any $R$-algebra $D$, the canonical morphism 
$$(A\times_{A/I} R)\otimes_R D\to (A\otimes_RD)\times_{(A\otimes_RD)/I_D} D$$ is an isomorphism, where $I_D:=\ker(A\otimes_DR\to (A/I)\otimes_RD).$ \label{basechangelemma}
\end{lemma}
\begin{proof}
Put $A':=A\times_{A/I} R$ and consider the exact sequence
$$0\to A'\to A\to (A/I)/R\to 0$$ of $R$-modules. Let $D$ be an arbitrary $R$-algebra. Using our assumption on $(A/I)/R$, we obtain an exact sequence
$$0=\mathrm{Tor}^R_1((A/I)/R, D)\to A'\otimes_{R}D\to A\otimes_RD\to ((A/I)/R)\otimes_RD\to 0.$$
We also have an exact sequence
$$0\to (A\otimes_RD) \times_{A\otimes_RD/I_D}D\to A\otimes_RD\to (A\otimes_RD/I_D)/D\to 0.$$ Using the flatness assumptions from the lemma repeatedly, we obtain isomorphisms
$$(A\otimes_RD/I_D)/D\cong ((A/I)\otimes_RD)/D\cong ((A/I)/R)\otimes_RD.$$ This implies that $A'\otimes_RD$ is the kernel of the map $A\otimes_RD\to (A\otimes_RD/I_D)/D,$ so the claim follows.
\end{proof}\\
\\
\noindent $\mathbf{Remark}.$ The requirement that both $A/I$ and $(A/I)/R$ be flat over $R$ in Lemma \ref{basechangelemma} is essential for the validity of the statement. Consider, for example, the case $R=\Og_K$, $A:=\Og_K[t]$, and the closed subscheme $Y:=\{0_k,1_k\}\subseteq \A^1_{\Og_K}.$ Suppose further that $Y$ is given by the ideal $I$. Then $(A/I)/R\cong k$ as $\Og_K$-modules. We find that the kernel of $(A\times_{A/I} k)\otimes_{\Og_K}k\to A\otimes_{\Og_K} k$ is given by $\mathrm{Tor}^{\Og_K}_1(k,k)=k$, which means that push-out and base change do not commute in this case. (Geometrically, $\Spec( A\times_{A/I} k)$ is the scheme obtained by identifying two points on the special fibre of $\A^1_{\Og_{K}}$.) Furthermore, we observe that $(A\times_{A/I} k)\otimes_{\Og_K} k$ is not reduced. To see this, note that we have an isomorphism
$$(A\times_{A/I} k)/\pi_K(A\times_{A/I} k)\cong (A\times_{A/I} k)\otimes_{\Og_K} k.$$ The ring $A\times_{A/I} k$ is explicitly given by the set of all polynomials $\sum_i a_it^i$ with the property that $\sum_{i\geq 1} a_i $ is divisible by $\pi_K$, and the ring structure inherited from $\Og_K[t].$ In particular, we see that $f:=\pi_K t-\pi_K\in A\times_{A/I} k.$ However, $f\not\in \pi_K\cdot(A\times_{A/I} k)$, which implies that the element of $(A\times_{A/I} k)\otimes_{\Og_K}k$ given by $f\otimes1$ is non-zero. On the other hand, a simple calculation shows that $f^2\in \pi_K \cdot(A\times_{A/I} k)$, so $(f\otimes 1)^2=0.$
\begin{lemma}
Keep the assumptions and the notation of the previous lemma, and moreover assume that $A/I$ and $(A/I)/R$ are projective $R$-modules. Let $D$ be an $A'$-algebra. Then the morphism \label{basechangelemmaII}
$$D\to (A\otimes_{A'}D)\times_{(A/I)\otimes_{A'}D}(R\otimes_{A'}D)$$ is surjective. If $D$ is a flat $A'$-algebra, then this map is an isomorphism. This Lemma remains valid even without assuming that $R$ be Noetherian.  
\end{lemma}
\begin{proof}
First note that the sequence of $A'$-modules
$$0\to R\to A/I\to (A/I)/R\to 0$$ is split exact, so the map $(A\otimes_{A'}D)\times_{(A/I)\otimes_{A'}D}(R\otimes_{A'}D)\to A\otimes_{A'}D$ is injective. Therefore, in order to prove the first claim, all we have to show is that an element $x\in A\otimes_{A'}D$ with the property that its image in $(A/I)\otimes_{A'}D$ comes from $R\otimes_{A'}D$ is of the form $x=1\otimes\delta$ for some $\delta\in D$. Choose $x\in A\otimes_{A'}D$ with this property, and denote its image in $(A/I)\otimes_{A'}D$ by $\overline{x}$. We find that 
$$\overline{x}-1\otimes \delta'=0$$ for some $\delta'\in D$. This is because we assume that the image of $x$ in $(A/I)\otimes_{A'}D$ comes from $R\otimes_{A'}D$. Since the map $A'\to R$ is surjective, every element of $R\otimes_{A'} D$ is an elementary tensor. We find that 
$$x-1\otimes\delta'\in \im(I\otimes_{A'}D\to A\otimes_{A'}D).$$ However, since $I\subseteq A'$, any element of $\im(I\otimes_{A'}D\to A\otimes_{A'}D)$ is of the form $1\otimes\delta$ for some $\delta\in D$. Putting things together, we find
$$x=1\otimes(\delta+\delta').$$
For the second statement, note that flatness implies that the map $D\to A\otimes_{A'}D,$ and hence the map from the Lemma, is injective. Also observe that, in this proof, we did not use the hypothesis that $R$ be Noetherian. 
\end{proof}
\begin{proposition}
Let $R$ be a Noetherian commutative ring with unity, and let $\tilde{\mathscr{C}}$ be a proper scheme over $R$. Let $Y$ be a closed subscheme of $\tilde{\mathscr{C}}$, finite and faithfully flat over $R$, which is contained in an open affine subscheme $\tilde{U}_0=\Spec \tilde{A}$ of $\tilde{\mathscr{C}}$, and given by the ideal $\tilde{I}$. Assume that $\tilde{A}/\tilde{I}$ and $(\tilde{A}/\tilde{I})/R$ are projective $R$-modules. Then the push-out
$$\mathscr{C}:=\tilde{\mathscr{C}}\cup_{Y}\Spec R$$ exists in the category of schemes over $R$ and is proper over $R$. Furthermore, let $X$ be an affine $R$-scheme, and put $\mathscr{C}_X:=\mathscr{C}\times_RX,$ and similarly for $\tilde{\mathscr{C}}.$ Also let $V\to \mathscr{C}_X$ be a flat morphism with $V$ affine. Put $\tilde{V}:=V\times_{\mathscr{C}_X}\tilde{\mathscr{C}}_X.$ Then the diagram
$$\begin{CD} \tilde{V}@>>>V\\
@AAA@AAA\\
V\times_{\mathscr{C}_X}(X\times_R\Spec(\tilde{A}/\tilde{I}))@>>>V\times_{\mathscr{C}_X}X
\end{CD}$$
is co-Cartesian in the category of affine schemes over $R$. \label{cocartesianproposition}
\end{proposition}
\begin{proof}
The first part of the Proposition follows from Proposition \ref{pushoutfiniteproposition}. For the second statement, choose an open affine covering $\tilde{U}_0,..., \tilde{U}_m$ of $\tilde{\mathscr{C}}$ with $\tilde{U}_0=\Spec\tilde{A},$ such that $\tilde{U}_i\cap Y=\emptyset$ for $i>0.$ By construction, we obtain an open affine covering $U_0,..., U_m$ of $\mathscr{C}$ with $U_0=\Spec (\tilde{A}\times_{\tilde{A}/\tilde{I}} R)$ and $\tilde{U}_i\cong  U_i$ as schemes for $i>0.$ By Lemma \ref{basechangelemma}, we may assume that $X=\Spec R.$ For each $i$, let $V_i$ be the inverse image of $U_i$ in $V$. Because $V$ and $\mathscr{C}$ are separated over $\Og_K$, the $V_i$ are affine schemes. From Lemma \ref{basechangelemmaII}, we know that the diagram
$$\begin{CD}\tilde{V}_0@>>> V_0\\
@AAA@AAA\\
V_0\times_{U_0} (\Spec \tilde{A}/\tilde{I})@>>>V_0\times_{U_0}\Spec R
\end{CD}$$
is co-Cartesian in the category of affine schemes over $R,$
where $\tilde{V}_i:=V_i\times_{U_i}\tilde{U}_i.$ But since $V_0\times_{U_0}\Spec R\cong V\times_{\mathscr{C}}\Spec R$ (and similarly for $\Spec (\tilde{A}/\tilde{I}$)), and the $\tilde{V}_i$ cover $\tilde{V}$, the result follows. 
\end{proof}\\
We shall also frequently employ the notion of \it cohomological flatness in dimension zero. \rm Recall that a morphism $f\colon X\to S$ is said to be cohomologically flat in dimension zero if for any $S$-scheme $s\colon S'\to S,$ the canonical morphism $$s^\ast f_\ast\Og_X\to (f\times_S\Id_{S'})_\ast\Og_{X\times_SS'}$$ is an isomorphism. In the case where $S=\Spec \Og_K$ and $f\colon X\to S$ is flat and proper, one has a relatively straightforward (and well-known) criterion for cohomological flatness in dimension zero:
\begin{proposition}
Let $f\colon X\to \Spec\Og_K$ be proper and flat. Then $f$ is cohomologically flat in dimension zero if and only if $H^1(X, \Og_X)$ is torsion-free. \label{cohomologicallyflatproposition}
\end{proposition}
\begin{proof}
(Compare \cite{Ca}, Lemma 2.1) Note that $H^1(X, \Og_X)$ is torsion-free if and only if $$H^1(X, \Og_X)[\pi_K]=0,$$ where $\pi_K$ denotes a uniformizing element of $\Og_K$. Let $\iota\colon X_k\to X$ denote the inclusion of the special fibre. Since $f$ is flat, we have an exact sequence 
$$0\to \Og_X\overset{\cdot\pi_K}{\to} \Og_X\to \iota_\ast\iota^\ast\Og_X\to 0.$$ Taking cohomology, we obtain the exact sequence
$$0\to \Gamma(X, \Og_X)\otimes_{\Og_K} k\to \Gamma(X_k, \iota^\ast\Og_X)\to H^1(X, \Og_X)[\pi_K]\to 0.$$ This implies that taking $f_\ast\Og_X$ commutes with the morphism $\Spec k\to \Spec\Og_K$ if and only if $H^1(X, \Og_X)[\pi_K]=0.$ But since taking $f_\ast\Og_X$ commutes with flat base change and $f$ is proper, this suffices to deduce the claim (see \cite{Liu}, Chapter 5.3, Corollary 3.22 and Remark 3.30).
\end{proof}\\
For the next lemma, recall that the \it index \rm of an algebraic curve $C$ over $K$ is given by the greatest common divisor of the numbers $[\kappa(x):K]$, where $\kappa(x)$ denotes the residue field of the point $x$ of $C$ and $x$ runs through all closed points of $C$.
\begin{lemma}
Let $\tilde{C}$ be a geometrically integral smooth proper curve over $K$. Let $\tilde{\mathscr{C}}\to\Spec\Og_K$ be a projective, flat, and regular model of $\tilde{C}$. Suppose further that $\tilde{C}$ has index one (or, equivalently, that $\tilde{C}$ admits a 0-cycle of degree 1). Then the morphism $\tilde{\mathscr{C}}\to \Spec\Og_K$ is cohomologically flat in dimension zero. \label{regularcohomologicallyflatlemma}
\end{lemma}
\begin{proof} 
Denote by $\tilde{\mathscr{C}}_k$ the special fibre of $\tilde{\mathscr{C}}\to \Spec\Og_K$. This may be viewed as a divisor on $\tilde{\mathscr{C}}$, which can be written as
$$\tilde{\mathscr{C}}_k=\sum_{i=1}^r \delta_i X_i,$$ with $\delta_j\in \N$, where the $X_j$ are the reduced irreducible components of $\tilde{\mathscr{C}}_k$. By \cite{Liu}, Chapter 9.1, Corollary 1.24, all we have to prove is that the greatest common divisor of the $\delta_i$ is equal to 1. Choose a 0-cycle $\sum_\nu \lambda_\nu z_\nu$ on $\tilde{C}$ of degree 1. For each $\nu$, let $Z_\nu$ be the Zariski closure of $z_\nu$ in $\tilde{\mathscr{C}}$. Then $Z_\nu$ is a horizontal divisor on $\tilde{\mathscr{C}}$. Now denote by $\langle-,-\rangle$ the intersection product of divisors on $\tilde{\mathscr{C}}.$ By \cite{Liu}, Chapter 9.1, Proposition 1.30, we have
$$\langle Z_\nu, \tilde{\mathscr{C}}_k\rangle=[\kappa(z_\nu):K]$$ for all $\nu.$ If we now define the divisor $Z$ to be equal to $\sum_\nu \lambda_\nu Z_\nu$, we obtain
\begin{align*}
1=\langle Z, \tilde{\mathscr{C}}_k\rangle =\sum_{i=1}^r \langle Z, X_i\rangle \delta_i,
\end{align*}
which implies the claim.
\end{proof}\\
\noindent$\mathbf{Remark.}$ From now on, we will always assume that the smooth curve $\tilde{C}$ have index one. This condition is clearly satisfied if $\tilde{C}$ admits a rational point, but there are smooth, proper, geometrically integral curves over $K$ which have index $>1$. For example, suppose that $E$ is an elliptic curve over $K$ with good reduction and that $P$ is a torsor for $E$ over $K$ whose class in $H^1_{\et}(\Spec K, E)$ has order $p$. Such an object exists by \cite{LLR}, Corollary 6.7, and its proof. From \cite{LLR}, Theorem 6.6, it follows that the special fibre of the minimal regular model of $P$ over $\Og_K$ is irreducible, but of multiplicity $p$. Using arguments from intersection theory (as in the proof of the preceding Lemma), we see that the index of $P$ must be divisible by $p.$

\section{Good models of singular curves}
In this section, we shall be concerned with constructing models with desirable properties of proper, geometrically integral curves $C$ over $K$, which have a \it push-out singularity. \rm For the sake of simplicity, let us assume that $C$ has precisely one singular point, and that this singularity is $K$-rational. More precisely, let $\tilde{C}$ be a smooth, proper, geometrically integral curve of index one over $K$. Suppose that $C$ fits into a co-Cartesian diagram

$$\begin{CD}
\tilde{C}@>>>C\\
@AAA@AAA\\
\Spec L@>>>\Spec K,
\end{CD}$$ where $L/K$ is a finite separable extension and the map $\Spec L\to \tilde{C}$ is a closed immersion. Such push-outs always exist by Proposition \ref{pushoutfiniteproposition}. Curve singularities which arise in this way will be called \it push-out singularities. \rm
If $L/K$ were inseparable, or more generally a non-étale algebra over $K$, then $\Pic^0_{C/K}$ would have subgroups of type $\mathbf{G}_a$, and would therefore not admit a Néron model. This is why we cannot consider push-outs along $K$-algebras which are not étale (see also the remark at the end of this paper). Recall that $\tilde{C}$ admits a proper, flat and regular model $\tilde{\mathscr{C}}$ over $\Og_K$.

\begin{proposition}
Let $\tilde{C}$ be a smooth, proper, geometrically integral curve of index one over $K$ and suppose that $C$ arises from $\tilde{C}$ as above. Then there exists a proper, flat, and regular model $\tilde{\mathscr{C}}$ of $\tilde{C}$, as well as a proper and flat scheme $\mathscr{C}\to\Spec\Og_K$, such that\\
(i) $\mathscr{C}\times_{\Og_K} \Spec K\cong C,$\\
(ii) the morphism $\psi\colon \tilde{C}\to C$ extends to a finite morphism $\tilde{\mathscr{C}}\to \mathscr{C}$ (which we shall also call $\psi$ by abuse of notation), and which fits into a push-out diagram
$$\begin{CD}\tilde{\mathscr{C}}@>{\psi}>>\mathscr{C}\\
@AAA@AAA\\
\Spec\Og_L@>>>\Spec\Og_K,
\end{CD}$$
(iii) the canonical morphism \label{singularmodelsproposition}
$$\Spec \Og_L\to \tilde{\mathscr{C}}\times_{\mathscr{C}} \Spec \Og_K$$ is an isomorphism, and\\
(iv) the morphism $\psi\colon \tilde{\mathscr{C}}\to \mathscr{C}$ is an isomorphism away from the image of $\Spec \Og_L$ in $\tilde{\mathscr{C}}$. \\
Furthermore, $\mathscr{C}$ is cohomologically flat in dimension zero.
\end{proposition}
\begin{proof}
Since $\tilde{C}$ is smooth, and since proper curves over fields are projective, the existence of a projective, flat, and regular model $\tilde{\mathscr{C}}$ of $\tilde{C}$ follows from \cite{Liu}, Chapter 10.1, Proposition 1.8. Since $\tilde{\mathscr{C}}$ is proper over $\Og_K$, the morphism $\Spec L\to \tilde{C}$  extends to a morphism
$$\Spec \Og_L \to \tilde{\mathscr{C}}.$$ The scheme-theoretic image $Y$ of this morphism is isomorphic to $\Spec R$ for some $\Og_K$-order $R\subseteq \Og_L$. Let $D$ be the effective divisor on $\tilde{\mathscr{C}}$ given by the scheme-theoretic image of the map $\Spec \Og_L\to \tilde{\mathscr{C}}$. By the embedded resolution theorem \cite{Liu}, Chapter 9.2, Theorem 2.26, we can find a projective birational morphism 
$$f\colon \tilde{\mathscr{C}}'\to\tilde{\mathscr{C}}$$ with $\tilde{\mathscr{C}}'$ regular, such that the divisor $f^\ast D$ has strict normal crossings. Since the reduced irreducible components of $f^\ast D$ are also regular (\cite{Liu}, Chapter 9.2, Remark 2.27), we may in particular assume that the morphism $\Spec\Og_L\to \tilde{\mathscr{C}}$ is a closed immersion. Now define $\mathscr{C}$ to be the push-out
$$\begin{CD}
\tilde{\mathscr{C}}@>>>\mathscr{C}\\
@AAA@AAA\\
\Spec\Og_L@>>>\Spec \Og_K.
\end{CD}$$
This push-out exists in the category of schemes by Proposition \ref{cocartesianproposition}, using the additional observation that the image of $\Spec \Og_L\to\tilde{\mathscr{C}}$ is contained in an open affine subscheme of $\tilde{\mathscr{C}}$. Properties (i), (ii), and (iv) from the Proposition now follow immediately from the construction (using Proposition \ref{pushoutfiniteproposition} and Lemma \ref{basechangelemma}), and condition (iii) follows because the morphism $\Spec \Og_L\to \tilde{\mathscr{C}}\times_{\mathscr{C}} \Spec \Og_K$ is a scheme-theoretically dominant closed immersion, and hence an isomorphism.
Now consider the exact sequence
$$0\to \Og_{\mathscr{C}}\to\psi_\ast\Og_{\tilde{\mathscr{C}}}\to \mathscr{F}\to 0,$$ where $\mathscr{F}:=(\psi_\ast\Og_{\tilde{\mathscr{C}}})/\Og_{\mathscr{C}}.$ We obtain an exact sequence
$$0\to \Gamma(\mathscr{C}, \mathscr{F})\to H^1(\mathscr{C}, \Og_{\mathscr{C}})\to H^1(\tilde{\mathscr{C}}, \Og_{\tilde{\mathscr{C}}})\to 0,$$ using that $\psi$ is finite. Now, if $i\colon \Spec \Og_K\to \mathscr{C}$ denotes the canonical closed immersion, we have 
$$\mathscr{F} \cong i_\ast (\Og_L/\Og_K),$$ so the global sections of this sheaf are clearly torsion-free. Because $f\colon \tilde{\mathscr{C}}\to \Spec\Og_K$ is cohomologically flat in dimension zero (this follows from Lemma \ref{regularcohomologicallyflatlemma}), $H^1(\tilde{\mathscr{C}}, \Og_{\tilde{\mathscr{C}}})$ is torsion-free as well (Proposition \ref{cohomologicallyflatproposition}). Hence so is $H^1(\mathscr{C}, \Og_{\mathscr{C}}).$ By Proposition \ref{cohomologicallyflatproposition}, the last claim follows, too.
\end{proof}
\section{The Picard functor}
It is well-known (\cite{BLR}, Chapter 9.5) that the Picard functor of a regular, flat, and proper curve $\tilde{\mathscr{C}}\to \Spec\Og_K$ which is cohomologically flat in dimension zero can be used to construct a Néron model of the Jacobian of $\tilde{\mathscr{C}}\times_{\Og_K} \Spec K$ (the result in \it loc. cit. \rm is more general; we use cohomological flatness in dimension zero to ensure that the Picard functor is representable by an algebraic space). The aim of this section is to show that this construction can be generalized to the case where the generic fibre of the curve $\mathscr{C}$ we consider has a push-out singularity. The main difficulty arises from the fact that the schemes we shall employ are in general very far from being normal. Throughout this chapter, $C$ will denote a proper, geometrically integral curve over $K$ with precisely one push-out singularity as before, $\tilde{C}$ will denote its normalization, and $\mathscr{C}$ and $\tilde{\mathscr{C}}$ will denote the schemes from Proposition \ref{singularmodelsproposition}. Furthermore, we shall always assume that $\tilde{C}$ has index one. This guarantees that the morphism $f\colon \tilde{\mathscr{C}}\to\Spec\Og_K$ is cohomologically flat in dimension zero (see Lemma  \ref{regularcohomologicallyflatlemma}).
Recall that the \it Picard Functor \rm $\Pic_{\mathscr{C}/\Og_K}$ of $\mathscr{C}\to \Spec\Og_K$ is defined to be the fppf-sheafification of the presheaf defined on $\Og_K$-schemes $X$ by
$$X\mapsto\Pic (\mathscr{C}\times_{\Og_K} X).$$ If we denote by $f^X$ the projection $\mathscr{C}\times_{\Og_K} X\to X$, then we have
$$\Pic_{\mathscr{C}/\Og_K}(X)=\Gamma(X, \R^1f^X_\ast\Gm)$$ (\cite{K}, p. 21), where the derived functor is taken with respect to the fppf-topology. Moreover, the Leray spectral sequence induces the low-degree exact sequence
\begin{align*}&H^1_{\fppf}(X, \Gm)\to H^1_{\fppf}(\mathscr{C}\times_{\Og_K} X, \Gm)\to \Pic_{\mathscr{C}/\Og_K}(X)\\ \to &H^2_{\fppf}(X, \Gm)\to H^2_{\fppf}(\mathscr{C}\times_{\Og_K} X, \Gm),\end{align*} and similarly with the fppf-topology replaced by the étale topology. The inclusion functor from the étale into the fppf-site induces a morphism between the exact sequence above and the corresponding exact sequence in the étale topology, and Grothendieck's theorem comparing étale and fppf-cohomology for smooth groups as well as the lemma of five homomorphisms implies that the Picard functor can  be defined using the étale instead of the fppf-topology as well. Also observe that we have a morphism of sheaves
$$\Pic_{\mathscr{C}/\Og_K}\to\Z$$ (and similarly for $\tilde{\mathscr{C}}$) coming from the fact that the degree of a line bundle on a flat family is locally constant.  
\begin{lemma}
There is an exact sequence of sheaves \label{exactsequencelemma}
$$0\to \mathscr{T}^0\to \Pic_{\mathscr{C}/\Og_K}\to \Pic_{\tilde{\mathscr{C}}/\Og_K}\to 0$$ on the big étale site of $\Spec\Og_K$, where $\mathscr{T}^0$ denotes the identity component of the Néron model of the torus $T:=(\Res_{L/K}\Gm)/\Gm.$ 
\end{lemma}
\begin{proof}
Let $X=\Spec D$ be an affine $\Og_K$-scheme. Put $\mathscr{C}_X:=\mathscr{C}\times_{\Og_K} X$ and similarly for $\tilde{\mathscr{C}}.$ Denote the structure morphisms of $\tilde{\mathscr{C}}_X$ and $\mathscr{C}_X$ by $\tilde{f}$ and $f$, respectively. The restriction of the functor $\Pic_{\mathscr{C}/\Og_K}$ to the small étale site of $X$ is given by $\R^1f_\ast\Og_{\mathscr{C}_X}^\times$, and similarly for $\tilde{\mathscr{C}}_X.$ Since the morphism $\psi\colon \tilde{\mathscr{C}}_X\to \mathscr{C}_X$ is scheme-theoretically dominant, we have an exact sequence of sheaves
\begin{align}0\to \Og_{\mathscr{C}_X}^\times\to \psi_\ast \Og_{\tilde{\mathscr{C}}_X}^\times\to\mathscr{Q}\to 0\label{exactsequence}\end{align} on the small étale site of $\mathscr{C}_X,$ where $\mathscr{Q}$ is defined to be the cokernel of the second map. Our first task will be to determine $\mathscr{Q}$. Denote the closed immersion $X\to \mathscr{C}_X$ by $\iota$ and put $X_L:=X\times_{\Og_K}\Spec\Og_L.$ We have canonical morphisms 
\begin{align}\psi_\ast\Og_{\tilde{\mathscr{C}}_X}^\times\to \iota_{\ast}\Res_{X_L/X}\Gm\to \iota_{\ast}((\Res_{X_L/X}\Gm)/\Gm)\label{epimorphism}.\end{align} Here, we denote by $\Res_{X_L/X}\Gm$ and $(\Res_{X_L/X}\Gm)/\Gm$ the sheaves on the small étale site of $X$ defined by these group schemes. Note that the second morphism is an epimorphism since $\iota_\ast-$ is exact. Also note that the first morphism is an epimorphism as well, as can be checked as follows: Let $U\to \mathscr{C}_X$ be étale with $U$ affine, and let $\tilde{U}:=U\times_{\mathscr{C}_X}\tilde{\mathscr{C}}_X.$ Let $\varphi \in \iota_{\ast} (\Res_{X_L/X}\Gm)(U)=\Res_{X_L/X}\Gm(U\times_{\mathscr{C}_X} X)=\Gm(\tilde{U}\times_{\tilde{\mathscr{C}}_X}X_L)$ (the last equality follows because $\tilde{U}\times_{\tilde{\mathscr{C}_X}}X_L=(U\times_{\mathscr{C}_X}\tilde{\mathscr{C}_X})\times_{\tilde{\mathscr{C}_X}}X_L=(U\times_{\mathscr{C}_X}X)\times_XX_L$).
We may extend $\varphi$ to a regular function $\tilde{\varphi}$ on $\tilde{U}$. Also extend $\varphi^{-1}$ to a regular function $\tilde{\mu}$ on $\tilde{U}$. Denote by $V(\tilde{\varphi}\tilde{\mu})$ the closed subscheme of $\tilde{U}$ on which $\tilde{\varphi}\tilde{\mu}$ vanishes. Then the intersection of $V(\tilde{\varphi}\tilde{\mu})$ and $\tilde{U}\times_{\tilde{\mathscr{C}}_X} X_L$ is empty, and the function $\tilde{\varphi}\tilde{\mu}$ comes from a regular function $\nu$ on $U$ by Proposition \ref{cocartesianproposition}. Now pick an open covering $W_0,W_1$ of $U$ by defining $W_0$ to be the open subset in $U$ on which $\nu$ does not vanish, and
$W_1$ to be the disjoint union of open sets which form an affine open cover of the complement of $U\times_{\mathscr{C}_X} X$ in $U$. Then the element of $\psi_\ast \Og_{\tilde{\mathscr{C}}_X}(W_0\cup W_1)$ given by $\tilde{\varphi}$ on $W_0$ and $1$ on $W_1$ restricts to $\varphi\in \iota_\ast \Res_{X_L/X}\Gm(W_0\cup W_1).$ Hence surjectivity follows.

We will now show that the kernel of the composition of the two morphisms (\ref{epimorphism}) is precisely $\Og_{\mathscr{C}_X}^\times.$ Let $U\to \mathscr{C}_X$ be étale with $U$ affine, and let $\tilde{U}:=\tilde{\mathscr{C}}_X\times_{\mathscr{C}_X} U.$  By Proposition \ref{cocartesianproposition}, we know that the diagram 
$$\begin{CD}
\tilde{U}@>>>U\\
@AAA@AAA\\
U\times_{\mathscr{C}_X}X_L@>>>U\times_{\mathscr{C}_X}X
\end{CD}$$
is a push-out diagram. Therefore, the image of any $\varphi\in \psi_\ast\Og_{\tilde{\mathscr{C}}_X}^\times(U)$ under the morphism (\ref{epimorphism}) vanishes if and only if the restriction of $\varphi$ to $U\times_{\mathscr{C}_X}X_L$ arises as the pull-back of an invertible function on $U\times_{\mathscr{C}_X}X.$ Because the diagram above is a push-out diagram, this happens if and only if $\varphi$ arises as the pull-back of an invertible function on $U$. Putting things together, we obtain
$$\mathscr{Q}\cong \iota_{\ast}((\Res_{X_L/X}\Gm)/\Gm).$$
Taking cohomology of the exact sequence (\ref{exactsequence}), we obtain an exact sequence
$$0\to \mathscr{T}^0\times_{\Og_K} X\to \R^1f_\ast\Og_{\mathscr{C}_X}^\times\to \R^1\tilde{f}_\ast\Og_{\tilde{\mathscr{C}}_X}^\times\to 0$$ on the small étale site of $X$. This follows because both $f_\ast\Og_{\mathscr{C}_X}^\times$ and $\tilde{f}_\ast\Og_{\tilde{\mathscr{C}_X}}^\times$ are canonically isomorphic to $\Gm$ by cohomological flatness in dimension zero, and because we have
$$\R^1f_\ast (\iota_\ast (\Res_{X_L/X}\Gm)/\Gm)\cong \R^1(f\circ\iota)_\ast((\Res_{X_L/X}\Gm)/\Gm)=0,$$ since $\iota_\ast-$ is exact. Furthermore, we have used Corollary \ref{torusneroncorollary} (i). It is clear from the construction that these exact sequences can be patched together (using base change morphisms in étale cohomology and the fact that a morphism of sheaves in the étale topology is determined by its restriction to the site of affine schemes with the étale topology) to give rise to the exact sequence from the Lemma.
\end{proof}\\
Now define $P_{\mathscr{C}}$ and $P_{\tilde{\mathscr{C}}}$ to be the kernel of $\Pic_{\mathscr{C}/\Og_K}\to \Z$ and $\Pic_{\tilde{\mathscr{C}}/\Og_K}\to \Z$, respectively.  
From Lemma \ref{exactsequencelemma} and Artin's representability theorem (\cite{Ar}, Theorem 7.3) we may deduce
\begin{proposition}
(i) We have an exact sequence \label{Pexactsequenceproposition}
$$0\to \mathscr{T}^0\to P_{\mathscr{C}}\to P_{\tilde{\mathscr{C}}}\to 0$$
of sheaves on the big étale site of $\Spec \Og_K$.\\
(ii) The sheaves $P_{\mathscr{C}}$ and $P_{\tilde{\mathscr{C}}}$ are algebraic spaces locally of finite type over $\Og_K$. In particular, the sequence above can be viewed as an exact sequence on the fppf-site of $\Spec \Og_K.$
\end{proposition}
\begin{proof}
Part (i) immediately follows from Lemma \ref{exactsequencelemma} together with the observation that the degree functions $\Pic_{\mathscr{C}/\Og_K}\to \Z$ and $\Pic_{\tilde{\mathscr{C}}/\Og_K}\to \Z$ are compatible (the argument in \cite{BLR}, Chapter 9.1, pp. 236ff implies that this holds generically, so it must hold globally as the generic fibres of $\Pic_{\mathscr{C}/\Og_K}$, $\Pic_{\tilde{\mathscr{C}}/\Og_K}$ are scheme-theoretically dense in both of these group spaces). For part (ii), recall that $\Pic_{\mathscr{C}/\Og_K}$ and $\Pic_{\tilde{\mathscr{C}}/\Og_K}$ are algebraic spaces locally of finite type over $\Og_K$, since $\mathscr{C}$ and $\tilde{\mathscr{C}}$ are cohomologically flat in dimension zero over $\Og_K$ (the former by Proposition \ref{singularmodelsproposition}; the latter because $\tilde{C}$ has index one; see Lemma \ref{regularcohomologicallyflatlemma}). This follows from Artin's Theorem \cite{Ar}, Theorem 7.3 (see also \cite{K}, Theorem 4.18.6). Since $P_{\mathscr{C}}$ is an open subfunctor of $\Pic_{\mathscr{C}/\Og_K}$ (and similarly for $P_{\tilde{\mathscr{C}}}$), the functors $P_{\mathscr{C}}$ and $P_{\tilde{\mathscr{C}}}$ are algebraic spaces locally of finite type over $\Og_K$ as well. 
\end{proof}\\
\\
\noindent $\mathbf{Remark.}$ In this paper, all algebraic spaces $\mathscr{X}\to S$ over a scheme $S$ are \it locally separated \rm and \it quasi-separated, \rm i.e. we assume that the diagonal morphism $\mathscr{X}\to \mathscr{X}\times_S\mathscr{X}$ is an immersion and quasi-compact. \\
\\
In order to deduce results about Néron models, one needs to pass to the maximal separated quotients of the functors $P_{\mathscr{C}}$ and $P_{\tilde{\mathscr{C}}}$. Let $E_{\mathscr{C}}$ and $E_{\tilde{\mathscr{C}}}$ denote the scheme-theoretic closures of the unit section in $P_{\mathscr{C}}$ and $P_{\tilde{\mathscr{C}}},$ respectively (see \cite{R}, Chapter 3). By \cite{R}, Proposition 3.3.5, $E_{\mathscr{C}}$ and $E_{\tilde{\mathscr{C}}}$ are group objects in the category of algebraic spaces locally of finite type over $\Og_K$, and they are étale over $\Og_K$. We define
$$P_{\mathscr{C}}^{\mathrm{sep}}:=P_{\mathscr{C}}/E_{\mathscr{C}},$$ and similarly for $\tilde{\mathscr{C}}.$ The quotient is taken with respect to the fppf-topology. By \cite{R}, Proposition 3.3.5, both $P_{\mathscr{C}}^{\mathrm{sep}}$ and $P_{\tilde{\mathscr{C}}}^{\mathrm{sep}}$ are representable by separated group schemes locally of finite type over $\Og_K.$ We need the following 
\begin{lemma}
The algebraic spaces $E_{\tilde{\mathscr{C}}}$ and $E_{\mathscr{C}}$ are schemes which admit an open covering by affine schemes isomorphic to $\Spec\Og_K$, glued together along the generic fibre.  \label{Schemeslemma}
\end{lemma}
\begin{proof}
Choose an étale cover $U\to E_{\mathscr{C}}.$ Since the generic fibre of $E_{\mathscr{C}}$ is trivial, we may assume that $U=\cup_{i\in I} \Spec\Og_K$ for some index set $I$. This is true since $\Spec \Og_K$ is strictly Henselian. Each map $\Spec\Og_K\to U\to E_{\mathscr{C}}$ can be seen as a section of the structural morphism $E_{\mathscr{C}}\to \Spec\Og_K$, and is therefore a monomorphism. Since each such map is also étale, it is an open immersion. This implies the claim.  
\end{proof}\\
We have
\begin{proposition}
(i) The morphism $P_{\mathscr{C}}^{\mathrm{sep}}\to P_{\tilde{\mathscr{C}}}^{\mathrm{sep}}$ of fppf-sheaves is surjective. \\
(ii) Let $\mathscr{K}$ be the kernel of $P_{\mathscr{C}}^{\mathrm{sep}}\to P_{\tilde{\mathscr{C}}}^{\mathrm{sep}}$. Then there is an étale (but not necessarily separated) group scheme $\mathscr{E}$ of finite type over $\Og_K$ such that there is an exact sequence
$$0\to \mathscr{T}^0\to\mathscr{K}\to\mathscr{E}\to 0.$$ In particular, $\mathscr{K}$ is of finite type and smooth over $\Og_K.$\label{Eproposition}
\end{proposition}
\begin{proof}
The first statement is clear since $P_{\mathscr{C}}\to P_{\tilde{\mathscr{C}}}\to P_{\tilde{\mathscr{C}}}^{\mathrm{sep}}$ is surjective. For the second statement, consider the commutative diagram
$$\begin{CD}
&&&&0&&0\\
&&&&@VVV@VVV\\
&&&&E_{\mathscr{C}}@>>>E_{\tilde{\mathscr{C}}}\\
&&&&@VVV@VVV\\
0@>>>\mathscr{T}^0@>>>P_{\mathscr{C}}@>>>P_{\tilde{\mathscr{C}}}@>>>0\\
&&@VVV@VVV@VVV\\
0@>>>\mathscr{K}@>>>P_{{\mathscr{C}}}^{\mathrm{sep}}@>>>P_{\tilde{\mathscr{C}}}^{\mathrm{sep}}@>>>0.\\
\end{CD}$$
It follows from Lemma \ref{Schemeslemma} and its proof that the morphism $E_{\mathscr{C}}\to E_{\tilde{\mathscr{C}}}$ is flat. This implies that the kernel of $E_{\mathscr{C}}\to E_{\tilde{\mathscr{C}}}$ is flat over $\Og_K$. But this kernel is a closed subgroup scheme of $\mathscr{T}^0$, which is separated.  This implies that the kernel is trivial. Now define $\mathscr{E}$ to be the cokernel of $E_{\mathscr{C}}\to E_{\tilde{\mathscr{C}}}$. From the description of $E_{\tilde{\mathscr{C}}}$ and $E_{\mathscr{C}}$ in the proof of Lemma \ref {Schemeslemma} it follows that $\mathscr{E}$ is representable by a group scheme which is étale over $\Og_K$. To show that it is of finite type, it suffices to show that the group $\mathscr{E}(\Og_K)$ is finite. To see this, observe that $E_{\tilde{\mathscr{C}}}(\Og_K)$ is the Abelian group generated by the line bundles $\Og_{\tilde{\mathscr{C}}}(-D)$, where $D$ is a reduced irreducible component of the special fibre of $\tilde{\mathscr{C}}$. Here we use that all elements of $\Pic_{\tilde{\mathscr{C}}/\Og_K}(\Og_K)$ are represented by line bundles; see \cite{BLR}, Chapter 9.1, Corollary 12, and that the same statement holds for $\mathscr{C}$, since $\mathscr{C}$ has a section over $\Og_K$. Hence, $E_{\tilde{\mathscr{C}}}(\Og_K)$ is finitely generated as an Abelian group, so it suffices to show that $\mathscr{E}(\Og_K)$ is torsion. To prove this, let $\mathscr{L}$ be a line bundle on $\tilde{\mathscr{C}}$ which is given by $\Og_{\tilde{\mathscr{C}}}(-D)$ as above. Choose an open cover $U_1,..., U_d$ of the complement of the image of $\Spec\Og_L\to \tilde{\mathscr{C}}$ on which $\mathscr{L}$ is trivial. Then choose an open neighbourhood $V$ of the image of $\Spec \Og_L\to\tilde{\mathscr{C}}$ on which $\mathscr{L}$ is trivial. Choose regular functions $f_j$ on $U_j$ and $g$ on $V$ such that $\mathscr{L}$ is the line bundle associated to the Cartier divisor defined by those rational functions. Now observe that the map $\psi\colon \tilde{\mathscr{C}}\to \mathscr{C}$ is a homeomorphism on the underlying topological spaces (it is certainly bijective, and closed since it is finite). Hence we obtain an open covering $\psi(U_1),...,\psi(U_d),\psi(V)$ of $\mathscr{C}$. The functions $f_j$ and $g$ define rational functions on this open cover, since $\psi$ is birational. One checks immediately that we obtain a Cartier divisor on $\mathscr{C}$ defined by those rational functions. If $\mathscr{M}$ is the line bundle on $\mathscr{C}$ that arises from this Cartier divisor, we have $\psi^\ast\mathscr{M} \cong\mathscr{L}$ by construction, but it is not in general true that $\mathscr{M}$ is trivial at the generic fibre. Since $\mathscr{L}$ is an ideal sheaf which is trivial at the generic fibre, we may assume that the $f_j$ and $g$ are invertible regular functions on the restrictions of the $U_j$ and $V$ to the generic fibres. Now let $n:=[L:K].$ The pull-back of $g$ along $\iota\colon \Spec \Og_{L}\to \tilde{\mathscr{C}}$ may be seen as a non-zero element of $\Og_{L}$. Therefore, there is a unit $c\in \Og_{L}^\times$ such that $\iota^\ast g^n=c\pi_K^{\delta}$ for some $\delta\in\N_0$. After shrinking $V$ if necessary, we may assume that $c$ extends to an invertible regular function $c$ on $V$. Now the Cartier divisor on ${\mathscr{C}}$ defined by the functions $f_j^n$ on $\psi(U_j)$ and $c^{-1}g^n$ on $\psi(V)$ defines a line bundle on $\mathscr{C}$, which pulls back to $\mathscr{L}^{\otimes n}$ on $\tilde{\mathscr{C}}.$ Clearly, this line bundle is trivial on the generic fibre of $\mathscr{C}$. In particular, $\mathscr{L}^{\otimes n}$ lies in the image of $E_{\mathscr{C}}(\Og_K)\to E_{\tilde{\mathscr{C}}}(\Og_K).$ This implies that $\mathscr{E}(\Og_K)$ is $n$-torsion. 
\end{proof}\\
Before proving our first main result, we need the following
\begin{lemma}
Let $\mathscr{F}$, $\mathscr{G}$ and $\mathscr{H}$ be commutative group schemes which are separated, flat, and of finite type over $\Og_K$. Suppose that we have a morphism $f\colon \mathscr{F}\to\mathscr{G}$ and a closed immersion $g\colon \mathscr{F}\to \mathscr{H}$ of group schemes. Then the push-out $\mathscr{G}\oplus_{\mathscr{F}}\mathscr{H}$, taken in the category of fppf-sheaves of Abelian groups, is representable. 
\end{lemma}
\begin{proof}
Because the projection $\mathscr{G}\times_{\Og_K}\mathscr{H}\to \mathscr{H}$ is separated, the morphism
$$\mathscr{F}\to \mathscr{G}\times_{\Og_K}\mathscr{H}$$ given by $x\mapsto(-f(x), g(x))$ is a closed immersion. By a result of Anantharaman \cite{An}, Théorème 4C, p.53, the quotient $(\mathscr{G}\times_{\Og_K}\mathscr{H})/\mathscr{F}$ is representable. However, one checks easily that this quotient satisfies the universal property of the push-out.
\end{proof}
\begin{theorem}
Let $C$ be a proper, geometrically integral curve over $K$. Assume that $C$ has precisely one singular point, which arises as a push-out as before. Further denote by $\tilde{C}$ the normalization of $C$ and let $\mathscr{T}$, $\mathscr{N}$ and $\tilde{\mathscr{N}}$ denote the Néron models of the semiabelian varieties $T$, $\Pic^0_{C/K}$, and $\Pic^0_{\tilde{C}/K}$, respectively. Then the sequence $$0\to \mathscr{T}\to \mathscr{N}\to\tilde{\mathscr{N}}\to 0$$ of group schemes over $\Og_K$ is exact. \label{Neronexacttheorem}
\end{theorem}
\begin{proof}
Let $\mathscr{C}$ be the model of $C$ from Proposition \ref{singularmodelsproposition}. We have the exact sequence
$$0\to \mathscr{K}\to P_{\mathscr{C}}^{\mathrm{sep}}\to P_{\tilde{\mathscr{C}}}^{\mathrm{sep}}\to 0.$$ As before, let $\mathscr{T}$ denote the Néron model of the torus $T=(\Res_{L/K}\Gm)/\Gm,$ and let $\mathscr{K}$ and $\mathscr{E}$ be the group schemes from Proposition \ref{Eproposition}. Since $\mathscr{E}$ is trivial at the generic fibre, $\mathscr{K}$ is a smooth model of $T$, and $\mathscr{K}$ is clearly separated over $\Og_K$. Therefore, there is a unique morphism $\mathscr{K}\to \mathscr{T}$ extending the identification of the generic fibres. Now consider the exact sequence of fppf-sheaves
$$0\to\mathscr{T} \to P_{\mathscr{C}}^{\mathrm{sep}}\oplus_{\mathscr{K}}\mathscr{T}\to P_{\tilde{\mathscr{C}}}^{\mathrm{sep}}\to 0.$$ Because the group schemes $\mathscr{T},$ $\mathscr{K}$, and $P_{\tilde{\mathscr{C}}}^{\mathrm{sep}}$ are separated, flat, finite type over $\Og_K$, the push-out $P_{\mathscr{C}}^{\mathrm{sep}}\oplus_{\mathscr{K}}\mathscr{T}$ is representable by a group scheme of finite type over $\Og_K$ by the previous Lemma. By \cite{BLR}, Chapter 9.5, Theorem 4, we know that $P_{\tilde{\mathscr{C}}}^{\mathrm{sep}}$ is a Néron model of $\Pic^0_{\tilde{C}/K}$. By \cite{BLR}, Chapter 7.5, Proof of Proposition 1(b), it follows that $P_{\mathscr{C}}^{\mathrm{sep}}\oplus_{\mathscr{K}}\mathscr{T}$ is a Néron model of $\Pic^0_{C/K}$. 
\end{proof}
\subsection{A remark on Chai's conjecture}
Suppose that $0\to H \to B\to A\to 0$ is an exact sequence with $H$ a torus, $B$ semiabelian, and $A$ Abelian. Let $K'$ be an extension of $K$ over which $B$ acquires semiabelian reduction (see \cite{SGA7}, Théorème 3.6 for the Abelian case; in general, one first extends $K$ to ensure that $A$ has semiabelian reduction and then makes a further extension to split $H$). Denote the Néron models of $H$, $A$, $B$ by $\mathscr{H}$, $\mathscr{A}$, $\mathscr{B}$, respectively. Further, put $H':=H\times_K\Spec K'$, and denote the Néron model of $H'$ over $\Og_{K'}$ by $\mathscr{H}'$. Use analogous notation for $A$ and $B$. We have a natural morphism
$$\phi\colon (\Lie \mathscr{H})\otimes_{\Og_K}\Og_{K'}\to \Lie\mathscr{H}'.$$ Define the \it base change conductor \rm of $H$ to be
$$c(H):=\frac{1}{[K':K]}\mathrm{length}_{\Og_{K'}}(\coker\phi),$$ and similarly for $A$ and $B$. This number is independent of the choice of $K'.$ Chai (\cite{Chai}, \S 8.1) conjectured that in this case, we have
$$c(B)=c(H)+c(A).$$ This conjecture is known if $K$ has characteristic zero, and in some other special cases. However, the proofs are, at least in the wildly ramified case, rather involved; for example, the proof presented in \cite{CLN} for the characteristic zero case uses a Fubini property for motivic integrals. If the sequence $0\to \mathscr{H}\to\mathscr{B}\to \mathscr{A}\to 0$ of Néron models is exact, then we can give an elementary proof of Chai's conjecture for the semiabelian variety $B$:
\begin{proposition} Keep the notation from this subsection, and suppose that the sequence $0\to \mathscr{H}\to \mathscr{B}\to \mathscr{A}\to 0$ is exact. Then we have
$$c(B)=c(H)+c(B).$$
\end{proposition}
\begin{proof}
Note that, since $B$ acquires semistable reduction over $K'$, the torus $H'$ is split. Hence, the sequence $0\to \mathscr{H}'\to \mathscr{B}'\to \mathscr{A}'\to 0$ is exact. (This follows from the argument in \cite{CY}, Lemma 11.2. The result is stated there only for exact sequences of tori, but their proof works in the more general case as well.) On the level of Lie algebras, we obtain the commutative diagram
$$\begin{CD}
0@>>>(\Lie\mathscr{H})\otimes_{\Og_{K}}\Og_{K'}@>>>(\Lie\mathscr{B})\otimes_{\Og_{K}}\Og_{K'}@>>>(\Lie\mathscr{A})\otimes_{\Og_{K}}\Og_{K'}@>>>0\\
&&@VVV@VVV@VVV\\
0@>>>\Lie{\mathscr{H}}'@>>>\Lie{\mathscr{B}}'@>>>\Lie{\mathscr{A}}'@>>>0.
\end{CD}$$
Now the result follows from the additivity of lengths of modules in exact sequences, and the snake lemma. 
\end{proof}
\begin{corollary}
Let $C$ be a proper curve over $K$ with precisely one push-out singularity as before. Then Chai's conjecture is true for the semiabelian variety $\Pic^0_{C/K}.$ 
\end{corollary}
\begin{proof} This follows from the Proposition preceding the Corollary and Theorem \ref{Neronexacttheorem}.
\end{proof}
\subsection{The connected component of the Néron model}
Let us now turn to the question of identifying the identity component of the Néron model $\mathscr{N}$ of $\Pic^0_{C/K}.$ Recall that the subsheaf $\Pic^0_{\mathscr{C}/\Og_K}$ of $P_{\mathscr{C}}$ is given by 
\begin{align*}
&\Pic^0_{\mathscr{C}/\Og_K}(S)\\
=&\{ \psi\colon S\to P_{\mathscr{C}}\colon \psi\times_{\Og_K} \Id_{\Spec k} \text{ factors through the identity component of } P_{\mathscr{C}}\times_{\Og_K} \Spec k\}
\end{align*}
for all $\Og_K$-schemes $S$, and similarly for $\tilde{\mathscr{C}}.$ This makes sense because $P_{\mathscr{C}}\times_{\Og_K} \Spec k$ is a group object in the category of algebraic spaces over the field $k$, and hence representable by a scheme. 
\begin{lemma}
The sequence
$$0\to \mathscr{T}^0\to \Pic^0_{\mathscr{C}/\Og_K}\to \Pic^0_{\tilde{\mathscr{C}}/\Og_K}\to 0$$ is exact.
\end{lemma}
\begin{proof}
Let $S\to \Pic^0_{\tilde{\mathscr{C}}/\Og_K}$ be a morphism. Because $P_{\mathscr{C}}\to P_{\tilde{\mathscr{C}}}$ is surjective, we can find an fppf-cover $S'\to S$ such that $S'\to \Pic^0_{\tilde{\mathscr{C}}/\Og_K}\to P_{\tilde{\mathscr{C}}}$ factors through $P_{\mathscr{C}}.$ However, since $\mathscr{T}^0_k$ is connected, the morphism $S'_k\to (P_{\mathscr{C}})_k$ must factor through the identity component, which implies that $S'\to P_{\mathscr{C}}$ factors through $\Pic^0_{{\mathscr{C}}/\Og_K}.$ Clearly, the kernel of $\Pic^0_{\mathscr{C}/\Og_K}\to \Pic^0_{\tilde{\mathscr{C}}/\Og_K}$ is equal to the scheme-theoretic intersection of $\Pic^0_{\mathscr{C}/\Og_K}$ with $\mathscr{T}^0$ inside $P_{\mathscr{C}}.$ However, the morphism $\mathscr{T}^0\to P_{\mathscr{C}}$ factors through $\Pic^0_{\mathscr{C}/\Og_K}$ since $\mathscr{T}^0_k$ is connected, so this scheme-theoretic intersection equals $\mathscr{T}^0.$ This implies the claim.
\end{proof}
\begin{proposition}
Let $\mathscr{N}$ be the Néron model of $\Pic^0_{C/K}$. The canonical morphism
$$\Pic^0_{\mathscr{C}/\Og_K}\to \mathscr{N}^0$$ is an isomorphism. In particular, $\Pic^0_{\mathscr{C}/\Og_K}$ is representable by a scheme.
\end{proposition}
\begin{proof}
Consider the commutative diagram of fppf-sheaves
$$\begin{CD}
0@>>>\mathscr{T}^0@>>>\Pic^0_{\mathscr{C}/\Og_K}@>>>\Pic^0_{\tilde{\mathscr{C}}/\Og_K}@>>>0\\
&&@VVV@VVV@VVV\\
0@>>>\mathscr{T}@>>>\mathscr{N}@>>>\tilde{\mathscr{N}}@>>>0\\
&&@VVV@VVV@VVV\\
0@>>>\mathscr{E}_1@>>>\mathscr{E}_2@>>>\mathscr{E}_3@>>>0,
\end{CD}$$
with exact rows, where the $\mathscr{E}_j$ are defined to be the cokernels of the vertical maps. We already know that $\mathscr{E}_1$ and $\mathscr{E}_3$ are étale group schemes over $\Og_K$ with trivial generic fibre (the case of $\mathscr{E}_3$ follows because the map $\Pic^0_{\tilde{\mathscr{C}}/\Og_K}\to \tilde{\mathscr{N}}$ induces an isomorphism between $\Pic^0_{\tilde{\mathscr{C}}/\Og_K}$ and the identity component $\tilde{\mathscr{N}}^0$ of $\tilde{\mathscr{N}}$; see \cite{BLR}, Chapter 9.5, Theorem 4(b); the greatest common divisor of the multiplicities of the irreducible components of $\tilde{\mathscr{C}}_k$ is equal to 1 by the argument from the proof of Lemma \ref{regularcohomologicallyflatlemma}, and there is no difference between multiplicity and geometric multiplicity since the residue field $k$ is algebraically closed). In particular, the unit sections of $\mathscr{E}_1$ and $\mathscr{E}_3$ are open immersions. This implies that the morphism $\mathscr{E}_1\to\mathscr{E}_2$ is representable by open immersions. Therefore, the unit section of $\mathscr{E}_2$ is representable by open immersions as well. In particular, the morphism $$\Pic^0_{\mathscr{C}/\Og_K}\to\mathscr{N}$$ is representable by an open immersion. Since the special fibre of $\Pic^0_{\mathscr{C}/\Og_K}$ is connected by construction, the claim follows.
\end{proof}\\
For the next proposition, let $\Phi(-)$ denote the group of connected components of the special fibre of a group scheme over $\Og_K$.
\begin{proposition}
Keep the notation from the previous lemmata and propositions. The sequences \label{Identitycomponentexactproposition}
$$0\to \mathscr{T}^0\to\mathscr{N}^0\to\tilde{\mathscr{N}}^0\to 0$$ and
$$0\to \Phi(\mathscr{T})_{\mathrm{tors}}\to \Phi(\mathscr{N})_{\mathrm{tors}}\to \Phi(\tilde{\mathscr{N}})_{\mathrm{tors}}\to 0$$ are exact.
\end{proposition}
\begin{proof}
The first assertion follows immediately from the previous two results. It also follows immediately that $0\to \Phi(\mathscr{T})\to \Phi(\mathscr{N})\to \Phi(\tilde{\mathscr{N}})\to 0$ is exact. However, since all component groups are torsion in our case, the Proposition follows. 
\end{proof}
\section{Applications to jumps and motivic Zeta functions}
\subsection{Applications to jumps}
Let us now apply our results to the study of jumps and motivic Zeta functions of the semiabelian varieties $\Pic^0_{C/K}$, where $C$ has precisely one $K$-rational push-out singularity. 
If $A$ is the Jacobian of a smooth proper curve of index one over $K$, then the jumps are rational numbers, and they only depend on the special fibre of a minimal sncd-model of the curve. This was proved in most cases by Halle \cite{Ha} and generalized in \cite{HN}, Chapter 6.3. In what follows, we will generalize this statement to the case of curves with push-out singularities. 
Consider the following 
\begin{definition}
Let $0\to T\to B\to A\to 0$ be an exact sequence of semiabelian varieties over $K$, and denote by $\mathscr{T}(d)$, $\mathscr{B}(d)$, and $\mathscr{A}(d)$ the Néron (lft-)models of $T\times_K\Spec K(d)$, $B\times_K\Spec K(d)$, and $A\times_K\Spec K(d)$, respectively. Also let $e\in \N$. We say that the sequence is \it $e$-universally exact over $\Og_K$ \rm if for all $d$ prime to $ep$ the sequence of Néron models
$$0\to \mathscr{T}(d)\to\mathscr{B}(d)\to \mathscr{A}(d)\to 0$$ is exact. We say that the sequence is \it universally exact over $\Og_K$ \rm if we can take $e=1$. 
\end{definition}
\begin{proposition}
Let $e\in\N$. Let $0\to T\to B\to A\to 0$ be an exact sequence of semiabelian varieties over $K$ which is $e$-universally exact over $\Og_K$. Let $\mathcal{J}_B$ be the multiset of jumps of $B$, and similarly for $T$ and $A$. Then we have \label{jumpsunionproposition}
$$\mathcal{J}_B=\mathcal{J}_T\cup\mathcal{J}_A$$
as multisets.
\end{proposition}
\begin{proof}
Choose a tower $K\subseteq K(d_1)\subseteq K(d_2)\subseteq ...$ of finite extensions of $K$ with $d_j$ prime to $ep$ which is cofinal in the set of all finite tame extensions of $K$ of degree prime to $ep$. For each $d$ prime to $p$, let $\alpha_d\colon \N_{\leq \dim A}\to \N_0$ be the increasing function whose image are precisely the $d$-jumps of $A$. Similarly define functions $\beta_d\colon \N_{\leq \dim B}\to\N_0$ (for $\mathscr{B}(d)$) and $\tau_d$ (for $\mathscr{T}(d)$). Suppose $x\in \mathcal{J}_B.$ By definition, there is $i\in \{1,..., \dim B\},$ such that $x=\lim_{n\to \infty} \beta_{d_n}(i)/{d_n}.$ By Proposition \ref{djumpsproposition}, we have 
$$\beta_{d_n}(\N_{\leq \dim B})=\tau_{d_n}(\N_{\leq \dim T})\cup \alpha_{d_n}(\N_{\leq \dim A})$$ as sets. Hence we may assume, without loss of generality, that for each $1\leq i\leq \dim B$ there are infinitely many $n$ such that $\beta_{d_n}(i)= \alpha_{d_n}(l)$ for some $1\leq l\leq \dim A.$ This implies that 
$$\lim_{n\to\infty} \frac{\beta_{d_n}(i)}{d_n}=\lim_{n\to\infty} \frac{\alpha_{d_n}(l)}{d_n},$$ so the inclusion "$\subseteq$" of sets follows, and the other inclusion can be deduced completely analogously. Here, we used the fact that the jumps can be calculated by considering only finite tame extensions of $K$ of degree prime to $e$; the argument is the same as in the proof of Corollary \ref{inducedjumpscorollary}, or \cite{HN}, proof of Theorem 6.3.1.3. It is clear that the multiplicities also coincide, so the result follows. 
\end{proof}
\begin{theorem} Let $C$ be a proper, geometrically integral curve over the field $K$. Assume that $C$ has precisely one singular point, which we assume to be a push-out singularity as before. Also assume that the normalization $\tilde{C}$ of $C$ has index one. Then the jumps of the semiabelian variety $\Pic^0_{C/K}$ are rational. Moreover, they only depend on the degree of the extension $L/K$ and the special fibre of a minimal sncd-model of $\tilde{C}$.  
\end{theorem}
\begin{proof}
By Proposition \ref{jumpsunionproposition}, \cite{HN}, Theorem 6.3.1.3, and \cite{HN}, Corollary 6.3.1.5, all we have to show is that the exact sequence
$$0\to T\to \Pic^0_{C/K}\to \Pic^0_{\tilde{C}/K}\to 0$$ is $[L:K]$-universally exact over $\Og_K$. This follows, however, because the base change of the sequence above to $\Spec K(d)$ coincides with the sequence
$$0\to T\times_K\Spec K(d)\to \Pic^0_{C\times_K\Spec K(d)/K(d)}\to \Pic^0_{\tilde{C}\times_K\Spec K(d)/K(d)}\to 0.$$ Here we use that $\Pic^0_{C/K}\times_K\Spec K(d)=\Pic^0_{C\times_K\Spec K(d)/K(d)}$ (and similarly for $\tilde{C}$), that the curve $C\times_K\Spec K(d)$ still has a push-out singularity by Lemma \ref{basechangelemma}, and that $L\otimes_KK(d)$ is a field whenever $d$ is prime to $p[L:K]$. Further, we use that 
$$(\Res_{L/K}\Gm)/\Gm\times_K\Spec K(d)=(\Res_{L\otimes_KK(d)/K(d)}\Gm)/\Gm,$$ and Theorem \ref{Neronexacttheorem}.
\end{proof}
\subsection{Applications to motivic Zeta functions}
Let us finally turn to the study of the motivic Zeta function associated to the semiabelian variety $\Pic^0_{C/K}.$ First, let us recall some definitions and some basic properties; the main reference being \cite{HN}, Chapter 8.
We also have to define other invariants of a semiabelian variety $B$ over $K$. The definitions are the same as in \cite{HN}. Furthermore, we shall from now on assume that the extension $L/K$ along which the push-out is performed is purely wild (i.e. that $[L:K]$ is a power of $p$). 
\begin{definition}
(i) For each $d$ prime to $p$, define $u(B(d))$ to be the unipotent rank of $\mathscr{B}(d)^0_k$ (that is, the dimension of the unipotent part of this algebraic group), and $t(G(d))$ to be the toric rank of $\mathscr{B}(d)^0_k.$\\
(ii) Define the \it potential toric rank $t_{tame}(B)$ of $B$ \rm to be 
$$t_{tame}(B):=\max_{p\nmid d}\ t(B(d)).$$
\end{definition}
Observe that, if the sequence $0\to T\to B\to A\to 0$ of semiabelian varieties over $K$ is $e$-universally exact over $\Og_K$ (for any $e\in\N$), then we have
$$c_{tame}(B)=c_{tame}(T)+c_{tame}(A)$$ (see definition \ref{elementarydefinition}). This is a direct consequence of Proposition \ref{jumpsunionproposition}. Furthermore denote by $\Phi(\mathscr{B}(d))$ the group of connected components of $\mathscr{B}(d)_k$. Also recall the definition of the \it stabilization index \rm $e(\tilde{C})$ of a smooth proper curve $\tilde{C}$ over $K$: It is defined to be the lowest common multiple of the multiplicities of the principal components of the special fibre of a minimal sncd-model of $C$ (see \cite{HN}, Definition 4.2.2.3). Let us now generalize this definition to allow for curves with push-out singularities:
\begin{definition}
Let $C$ be a proper, geometrically integral curve over $K$ with precisely one singular point, which we assume to arise as a push-out as before (recall that we assume the finite extension $L/K$ to be purely wild). Denote by $\tilde{C}$ the normalization of $C$. Define the \it stabilization index \rm $e(C)$ of $C$ to be 
$$e(C):=\mathrm{lcm}(e(\tilde{C}), [L:K]).$$
\end{definition}
\begin{lemma} Let $C$ be as before, and let $p\nmid d$. Denote by $\mathscr{N}(d)$ and $\tilde{\mathscr{N}}(d)$ the Néron models of $\Pic^0_{C\times_K K(d)}$ and $\Pic^0_{\tilde{C}\times_KK(d)}$ over $\Og_{K(d)}$, respectively. Further denote by $\Phi(-)$ the group of connected components of a Néron model. Let $\alpha\in \{1,..., e(C)\}$ be prime to $p$ and define $\alpha':=\mathrm{gcd}(\alpha, e(C))=\mathrm{gcd}(\alpha, e(\tilde{C})).$ Then we have\\
(i) $e(C(\alpha'))=e(C)/\alpha',$\\
(ii) $[\mathscr{N}(\alpha+qe(C))_k^0)]=[\mathscr{N}(\alpha')_k^0]$ in $K_0(\mathrm{Var}_k),$\\
(iii) $\#\Phi(\mathscr{N}(\alpha+qe(C)))_{\mathrm{tors}}=((\alpha+qe(C))/\alpha')^{t(\Pic^0_{\tilde{C}/K}\times_KK(\alpha'))}\#\Phi(\mathscr{N}(\alpha'))_{\mathrm{tors}}.$
\end{lemma}
\begin{proof}
By definition and by \cite{HN}, Proposition 4.2.2.9, we have
\begin{align*}e(C(\alpha'))&=\mathrm{lcm}\Big(\frac{e(\tilde{C})}{\mathrm{gcd}(e(\tilde{C}),\alpha')},[L:K]\Big)\\
&=\frac{\mathrm{lcm}(e(\tilde{C}),[L:K])}{\mathrm{gcd}(e(\tilde{C}),\alpha')}\\&=e(C)/\alpha'.\end{align*}
For the second equation, we use that $\mathrm{gcd}(e(\tilde{C}), \alpha')$ and $[L:K]$ have no common factor since $L/K$ is assumed to be purely wild. This shows that (i) is true. For part (ii), consider the sequence
$$0\to \mathscr{T}(\beta)^0_k\to\mathscr{N}(\beta)^0_k\to\tilde{\mathscr{N}}(\beta)^0_k\to 0$$ for $\beta\in\N$ prime to $p,$ which is exact by Proposition \ref{Identitycomponentexactproposition}.  Since this exact sequence is a Zariski fibre bundle (because $H^1_{\mathrm{Zar}}(\tilde{\mathscr{N}}(\beta)^0_k,\mathscr{T}(\beta)^0_k)\to H^1_{\fppf}(\tilde{\mathscr{N}}(\beta)^0_k,\mathscr{T}(\beta)^0_k)$ is an isomorphism; see \cite{N}, Proof of Lemma 3.1), we find
$$[\mathscr{N}(\beta)^0_k]=[\mathscr{T}(\beta)^0_k][\tilde{\mathscr{N}}(\beta)^0_k]$$ in $K_0(\mathrm{Var}_k).$ This reduces the claim to proving the analogous statement for the torus $T$, since the case of Jacobians of smooth curves follows from \cite{HN}, Proposition 8.2.1.3. Observe further that, for all $\beta$ prime to $p$, we have
$$[\mathscr{T}(\beta)^0_k]=(\mathbf{L}-1)^{t(T(\beta))}\mathbf{L}^{\dim T-t(T(\beta))}.$$ Therefore, in order to prove statement (ii) for the toric part, we only have to show that the toric ranks of $\mathscr{T}(\alpha')^0_k$ and $\mathscr{T}(\alpha+qe(C))^0_k$ coincide. This, however, is clear since the toric rank of $\mathscr{T}(d)^0_k$ is equal to 0 for any $d$ prime to $p$. 
For part (iii), we observe that Corollary \ref{torusneroncorollary}(ii) implies
$$\#\Phi(\mathscr{T}(d))_{\mathrm{tors}}=[L:K]$$ for any $d$ prime to $p$. This implies that
$$\#\Phi(\mathscr{T}(\alpha+qe(C)))_{\mathrm{tors}}=\#\Phi(\mathscr{T}(\alpha'))_{\mathrm{tors}}.$$ Since we already know that 
$$\#\Phi(\tilde{\mathscr{N}}(\alpha+qe(C)))_{\mathrm{tors}}=((\alpha+qe(C))/\alpha')^{t(\Pic^0_{\tilde{C}/K}\times_KK(\alpha'))}\#\Phi(\tilde{\mathscr{N}}(\alpha'))_{\mathrm{tors}}$$ by \cite{HN}, Proposition 4.3.1.1, the result follows from Proposition \ref{Identitycomponentexactproposition}.
\end{proof}
\begin{lemma}
Let $C$ be as before. For each $d$ prime to $p$, put $\ord(d):=\ord_{\Pic^0_{C/K}}(d)$. Let $\alpha\in \{1,...,e(C)\}$ be prime to $p$. Then, for each $q\in \N$ such that $p\nmid(\alpha+qe(C))$, we have
$$\ord(\alpha+qe(C))=\ord(\alpha)+qe(C)c_{tame}(\Pic^0_{C/K}).$$
\end{lemma}
\begin{proof}
Clearly, the order function is additive in short exact sequences of semiabelian varieties over $K$ which are universally exact over $\Og_K$, so it suffices to treat the Abelian and toric parts separately. 
First note that the sequence
$$0\to \Gm\to \Res_{L/K}\Gm\to T\to 0$$ is universally exact over $\Og_K$ by \cite{CY}, Lemma 11.2. This implies that 
$$\ord_{T}(d)=\ord_{\Res_{L/K}\Gm}(d).$$ From Proposition \ref{torusorderfunctionproposition}, we know that
$$\ord_{T}(\alpha+q\frac{e(C)}{[L:K]}[L:K])=\ord_T(\alpha)+qe(C)c_{tame}(T),$$ and by \cite{HN}, Proposition 8.2.2.2, we know that
$$\ord_{\Pic^0_{\tilde{C}/K}}(\alpha+q\frac{e(C)}{e(\tilde{C})}e(\tilde{C}))=\ord_{\Pic^0_{\tilde{C}/K}}(\alpha)+qe(C)c_{tame}(\Pic^0_{\tilde{C}/K}).$$
These two observations together imply the result. 
\end{proof}
\begin{theorem} Let $C$ be as before. Then the motivic Zeta function $Z_{\Pic^0_{C/K}}(z)$ is a rational function. More precisely, it is contained in the subring \label{motivicrationaltheorem}
$$K_0(\mathrm{Var}_k)\Big[z, \frac{1}{1-\mathbf{L}^ az^b}\Big]_{(a,b)\in \Z\times\Z_{>0}\colon a/b=c_{tame}(\Pic^0_{C/K})} \subseteq K_0(\mathrm{Var}_k)[\![z]\!].$$ The function $Z_{\Pic^0_{C/K}}(\mathbf{L}^{-s})$ has a unique pole at $s=c_{tame}(\Pic^0_{C/K})$, whose order is equal to $t_{tame}(\Pic^0_{\tilde{C}/K})+~1.$
\end{theorem}
\begin{proof}
Using the lemmata preceding this theorem, the proof of Theorem 8.3.1.2 from \cite{HN} can be taken \it mutatis mutandis. \rm 
\end{proof}\\
\\
\noindent $\mathbf{Remark}.$ The theorem above shows in particular that the order of the pole of $Z_{\Pic^0_{C/K}}(\mathbf{L}^{-s})$ only depends on the Abelian part of $\Pic^0_{C/K}$. This was known previously only in the tamely ramified case (\cite{HN}, Theorem 8.3.1.2) and the case of Jacobians of smooth curves (\it loc. cit.\rm), which have no toric part. See also \cite{HN2}, Theorem 8.6 for the case of tamely ramified Abelian varieties. We also observe that $e(C)$ is the smallest positive integer such that $e(C)j\in \Z$ for all jumps $j$ of $\Pic^0_{C/K}.$
Many of the technical results and results about jumps of semiabelian Jacobians in this paper can be generalized to the case where the finite extension $L/K$ is replaced by a general finite étale algebra $A$ over $K$, or to the case where we allow the curve $C$ to have more than one push-out singularity. However, this leads to no fundamentally new insights, but most proofs would become rather cumbersome were one to make this generalization. However, the method we used to prove rationality of motivic Zeta functions does not seem to be generalizable in this way: The toric part of the Jacobian of a proper $K$-curve $C$ which arises as the push-out of a smooth $K$-curve along the morphism $\Spec A\to \Spec K$ (where $A$ is an étale $K$-algebra) is not anisotropic as soon as $A$ is not a field. This implies that the group of connected components of the Jacobian of such a curve is in general not torsion. Hence we cannot deduce from the exactness of the sequence of component groups the exactness of the sequence of the torsion parts of those component groups. This problem is related to the \it index \rm of semiabelian varieties; see \cite{HN}, Definition 5.2.1.3. It would be very interesting to know whether the index of a semiabelian Jacobian can be calculated in terms of the models of curves with push-out singularities which we constructed.\\
\\
$\mathbf{Acknowledgement}.$ The author would like to express his gratitude to Dr. Johannes Nicaise for suggesting that he think about the question of rationality of jumps for algebraic tori and more general semiabelian varieties, and for sharing many of his ideas. In addition, the author is grateful to Professor Alexei Skorobogatov and Dr. Lars Halvard Halle for reading an earlier version of this manuscript and for valuable suggestions, as well as to Dr. David Holmes for helpful discussions. The author would also like to thank the anonymous referees for making a number of comments which greatly improved the presentation of this paper. This work was supported by the Engineering and Physical Sciences Research Council [EP/L015234/1], and the EPSRC Centre for Doctoral Training in Geometry and Number Theory (London School of Geometry and Number Theory), University College London.

\end{document}